\documentclass[reqno]{amsart}

\usepackage{lmodern}
\usepackage[utf8]{inputenc}
\usepackage[english]{babel}
\usepackage{microtype}

\usepackage{amsmath,amssymb,amsthm,mathrsfs,latexsym,mathtools,mathdots,booktabs,enumerate,tikz,tikz-cd,bm,url}
\usepackage[centertableaux]{ytableau}

\usetikzlibrary{arrows,positioning,arrows.meta,shapes,decorations.markings,matrix}
\usepackage[capitalize,noabbrev]{cleveref}
\crefname{ineq}{inequality}{inequalities}
\creflabelformat{ineq}{#2{\upshape(#1)}#3}

\usepackage[colorinlistoftodos]{todonotes}
\presetkeys%
{todonotes}%
{inline,backgroundcolor=yellow}{}

\newtheorem{theorem}{Theorem}
\newtheorem{proposition}[theorem]{Proposition}
\newtheorem{lemma}[theorem]{Lemma}
\newtheorem{corollary}[theorem]{Corollary}
\newtheorem{conjecture}[theorem]{Conjecture}

\theoremstyle{definition}
\newtheorem{definition}[theorem]{Definition}
\newtheorem{example}[theorem]{Example}
\newtheorem{remark}[theorem]{Remark}

\newtheorem*{note}{Note}

\newcommand{\defin}[1]{\textcolor{blue}{\emph{#1}}}

\newcommand{\setN}{\mathbb{N}}
\newcommand{\setZ}{\mathbb{Z}}
\newcommand{\setC}{\mathbb{C}}

\newcommand{\xvec}{\mathbf{x}}

\newcommand{\powerSum}{\mathrm{p}}

\newcommand{\schurS}{\mathrm{s}}

\newcommand{\thsup}{\textnormal{th}}
\newcommand{\rdsup}{\textnormal{rd}}

\newcommand{\fbound}{B}

\newcommand{\SSYT}{\mathrm{SSYT}}
\newcommand{\SYT}{\mathrm{SYT}}
\newcommand{\SYTT}{\mathrm{SYTT}}
\newcommand{\BST}{\mathrm{BST}}

\newcommand{\COF}{\mathrm{COF}}

\newcommand{\GL}{\mathrm{GL}}
\newcommand{\symS}{\mathfrak{S}}
\newcommand{\gl}{\mathfrak{gl}}

\DeclareMathOperator{\st}{st} 
\DeclareMathOperator{\shape}{sh}
\DeclareMathOperator{\Hom}{Hom}
\DeclareMathOperator{\End}{End}
\newcommand{\Perm}{\mathcal{P}} 
\newcommand{\Alt}[2]{\mathcal A_{#2}^{(#1)}} 
\newcommand{\alt}{\mathcal{A}} 
\newcommand{\wgt}{\mu} 
\DeclareMathOperator{\rot}{rot} 
\DeclareMathOperator{\pr}{pr} 

\def\longcycle{(1,\dots,n)}
\def\tocyclic{\!\downarrow_{\langle\longcycle\rangle}} 

\DeclareMathOperator{\height}{height}
\DeclareMathOperator{\maj}{maj}

\newcommand{\qbinom}{\genfrac{[}{]}{0pt}{}}

\tikzset{iTrio/.pic={%
\begin{scope}[yscale=-1]
\draw[black,x=1em,y=1em] (0,0)--(1,0)--(1,3)--(0,3)--(0,0);
\end{scope}
}}
\tikzset{hTrio/.pic={%
\begin{scope}[yscale=-1]
\draw[black,x=1em,y=1em] (0,0)--(0,1)--(3,1)--(3,0)--(0,0);
\end{scope}
}}
\tikzset{jTrio/.pic={%
\begin{scope}[yscale=-1]
\draw[black,x=1em,y=1em] (1,0)--(2,0)--(2,2)--(0,2)--(0,1)--(1,1)--(1,0);
\end{scope}
}}
\tikzset{tTrio/.pic={%
\begin{scope}[yscale=-1]
\draw[black,x=1em,y=1em] (0,0)--(2,0)--(2,1)--(1,1)--(1,2)--(0,2)--(0,0);
\end{scope}
}}

\title{Skew characters and cyclic sieving}
\thanks{
Alexandersson was partly supported by the Swedish Research Council (Vetenskapsrådet), grant 2015-05308.
Pfannerer and Rubey were supported by the the Austrian Science Fund (FWF): P 29275. Pfannerer is a recipient of a DOC Fellowship of the Austrian Academy of Sciences.}

\author{Per Alexandersson}
\email{per.w.alexandersson@gmail.com}
\address{Department of Mathematics,
Stockholm University,
S-10691, Stockholm, Sweden}

\author{Stephan Pfannerer}
\email{stephan.pfannerer@tuwien.ac.at}
\address{Fakultät für Mathematik und Geoinformation, TU Wien, Austria}

\author{Martin Rubey}
\email{martin.rubey@tuwien.ac.at}
\address{Fakultät für Mathematik und Geoinformation, TU Wien, Austria}

\author{Joakim Uhlin}
\email{joakim\_uhlin@hotmail.com}
\address{Department of Mathematics,
Stockholm University,
S-10691, Stockholm, Sweden}

\begin{document}

\begin{abstract}
  In 2010, B.~Rhoades proved that promotion together with the
  fake-degree polynomial associated with rectangular standard Young
  tableaux give an instance of the cyclic sieving phenomenon.

  We extend this result to all skew standard Young tableaux where the
  fake-degree polynomial evaluates to nonnegative integers at roots
  of unity, albeit without being able to specify an explicit group
  action.  Put differently, we determine in which cases a skew
  character of the symmetric group carries a permutation
  representation of the cyclic group.

  We use a method proposed by N.~Amini and the first author, which
  amounts to establishing a bound on the number of border-strip
  tableaux of skew shape.

  Finally, we apply our results to the invariant theory of tensor
  powers of the adjoint representation of the general linear group.
  In particular, we prove the existence of a bijection between
  permutations and J.~Stembridge's alternating tableaux, which
  intertwines rotation and promotion.
\end{abstract}

\maketitle

\setcounter{tocdepth}{1}
\tableofcontents

\pagebreak
\section{Introduction}
We determine which tensor powers of a skew character $\chi^{\lambda/\mu}$ of the symmetric group $\symS_n$ carry a permutation representation of the cyclic group of order $n$.

This problem can be rephrased in terms of V.~Reiner, D.~Stanton \&
D.~White's \emph{cyclic sieving
  phenomenon}~\cite{ReinerStantonWhite2004}.  Let $\SYT(\lambda/\mu)$
be the set of standard Young tableaux of skew shape $\lambda/\mu$,
and let $f^{\lambda/\mu}(q)$ be G.~Lusztig's fake degree polynomial
for $\chi^{\lambda/\mu}$.

Then there \emph{exists} an action $\rho$ of the cyclic group of
order $n=|\lambda/\mu|$ such that
\begin{equation*}
  \left(\underbrace{\SYT(\lambda/\mu)\times\dots\times\SYT(\lambda/\mu)}_m, \langle \rho \rangle, %
    f^{\lambda/\mu}(q)^m\right)
\end{equation*}
exhibits the cyclic sieving phenomenon, if and only if
$f^{\lambda/\mu}$ evaluates to nonnegative integers at $n^\thsup$ root
of unity.
If $m$ is even this is always the case.  If $m$ is odd, this is the
case if and only if there exists a tiling of $\lambda/\mu$ with
border-strips of size $k$ of even height for every $k\mid n$, see \cref{prop:positive-cyclic-sieving}.

We also show that for any skew shape $\lambda/\mu$ and integer $s>0$ there is an action $\tau$ of the cyclic group of order $s$ on \emph{stretched shapes} such that
\begin{equation*}
  \left(\SYT(s\lambda/s\mu), \langle \tau \rangle, %
    f^{s\lambda/s\mu}(q)\right)
\end{equation*}
exhibits the cyclic sieving phenomenon, see \cref{thm:stretchedCSPtheorem}.

At this point we are unable to present $\rho$ and $\tau$ explicitly for general skew shapes $\lambda/\mu$.
Instead, we use a characterization of P.~Alexandersson \& N.~Amini~\cite{AlexanderssonAmini2018}, which says that $f\in\setN[q]$ is a cyclic sieving polynomial for a group action of the cyclic group of order $n$, if and only if for a primitive $n^\thsup$ root of unity $\xi$ and all $k\mid n$ we have that $f(\xi^k)\in\setN$ and
\begin{equation*}
\sum_{d\mid k}\mu(k/d)f(\xi^d)\geq 0,
\end{equation*}
where $\mu$ is the number-theoretic Möbius function.

To apply this result, our main tool is a new bound for the absolute value of the skew
character evaluated at a power of the long cycle.  More precisely, with \cref{thm:BST-bound-skew} we
show that for any $k\mid n$
\begin{equation*}
  |f^{\lambda/\mu}(\xi^k)| \ge \sum_{d\mid k, d<k} |f^{\lambda/\mu}(\xi^d)|
\end{equation*}
provided $|f^{\lambda/\mu}(\xi^k)|\ge 2$.

To do so, we note that
$|f^{\lambda/\mu}(\xi^d)| = |\chi^{\lambda/\mu}((m^d))| =
|\BST(\lambda/\mu,m)|$, the number of border-strip tableaux of shape
$\lambda/\mu$ with strips of size $m$, extending the theorems for
straight shapes by T.~Springer~\cite{Springer1974} and G.~James \&
A.~Kerber~\cite{JamesKerber1984}.  We finally approximate the number
of border-strip tableaux using a bound by S.~Fomin \&
N.~Lulov~\cite{FominLulov1997}.

Our main motivation is an implication for the invariant theory of the
general linear group, as we now explain.  Let $\gl_r$ be the adjoint
representation of $\GL_r$, and consider its $n^\thsup$ tensor power
$\gl_r^{\otimes n}$.  The symmetric group $\symS_n$ acts on this space
by permuting tensor positions.  Thus, using Schur--Weyl duality, we
can determine the subspace of $\GL_n$-invariants of
$\gl_r^{\otimes n}$, regarded as a representation of $\symS_n$.  It
turns out to be isomorphic to
\[
  \bigoplus_{\substack{\lambda\vdash n\\\ell(\lambda)\leq r}} S_\lambda\otimes S_\lambda,
\]
where the direct sum is over all partitions of $n$ into at most $r$
parts, and $S_\lambda$ is the irreducible representation of $\symS_n$
corresponding to $\lambda$.  In particular, for $r\geq n$, the
dimension of the space of invariants equals the size of $\symS_n$.

A fundamental question of invariant theory is to find an explicit
basis of the space of invariants, and if possible enjoying further
desirable properties.  One such property is invariance under rotation
of tensor positions, following G.~Kuperberg's idea of web bases~\cite{Kuperberg1996}.

An elegant and useful solution would be to describe a set of
permutations in $\symS_n$ and a bijection from these to the basis
elements which intertwines rotation of permutations (that is,
conjugation with the long cycle) and rotation of tensor positions.
It would be even nicer if the set of permutations for the invariants
of $\gl_r^{\otimes n}$ would be a subset of the set of permutations
for the invariants of $\gl_{r+1}^{\otimes n}$.

Although it appears to be difficult to exhibit such an intertwining
bijection explicitly, our results , combined with previous work
of S.~Pfannerer, M.~Rubey \& B.~Westbury~\cite{PfannererRubeyWestbury2020},
implies that such a solution must exist, see \cref{thm:existence-pr-rot}.

The existence of such an intertwining bijection is closely related to the existence of a rotation invariant statistic $\st$ mapping permutations to partitions, such that $|\{\sigma\in \symS_n: \st(\sigma)=\lambda\}| = |\SYT(\lambda)\times \SYT(\lambda)|$, see \cref{cor:rotationInvatiantStat}.

\subsection{Outline of the paper}

In \cref{sec:csp} we recall the definition of the cyclic
sieving phenomenon and establish the connection with characters of cyclic group actions.
In \crefrange{sec:schur}{sec:borderStrips} we generalize T.~Springer's theorem to skew shapes and show, using the Murnaghan--Nakayma rule, the abacus of G.~James \& A.~Kerber and the Littlewood map, that the character evaluation of a skew character is, up to sign,
equal to a certain number of border-strip tableaux.
We stress that these identities are known for the straight shape case.
However, they are somewhat underappreciated gems
which deserve more attention.

In \cref{sec:BSTBounds} we provide the crucial bound on the number of border-strip
tableaux of given shape, building on the approximation
of S.~Fomin \& N.~Lulov.  In \cref{sec:mainCSP} we use this bound and
the characterization of P.~Alexandersson \& N.~Amini to prove the
existence of the group actions announced above for skew standard tableaux.

Finally in \cref{sec:gln} we apply our results to permutations and the invariant theory of the adjoint representation of the general linear group.

\section{Cyclic group actions and cyclic sieving}\label{sec:csp}

In this section we recall V.~Reiner, D.~Stanton \& D.~White's cyclic sieving
phenomenon, characters of cyclic group actions and a result of
P.~Alexandersson \& N.~Amini characterizing characters arising from cyclic
group actions.  We also recall R.~Brauer's permutation lemma, which
guarantees that two actions of the cyclic group which have the same
character as linear representations are even isomorphic as group actions.

\begin{definition}[\cite{ReinerStantonWhite2004}]
  Let $X$ be a finite set and let $\rho$ be a generator of an
  action of the cyclic group of order $n$ on $X$.

  Given a polynomial $f(q)\in \setN[q]$ we say that the triple
  $(X,\langle\rho\rangle,f(q))$ \defin{exhibits the cyclic sieving
    phenomenon} if for all $d \in \setZ$
  \begin{align}\label{eq:cspDef}
    \#\{ x\in X : \rho^d \cdot x = x \} = f(\xi^d),
  \end{align}
  where $\xi$ is a primitive $n^\thsup$ root of unity.  In this case
  $f(q)$ is a \defin{cyclic sieving polynomial} for the group
  action.
\end{definition}
In particular, the cardinality of $X$ is given by $f(1)$.
More generally, identifying the ring of
characters of the cyclic group of order $n$ with $\setZ[q]/(q^n-1)$,
the cyclic sieving polynomial $f(q)$ reduces to the character of the
group action modulo $q^n-1$.

The cyclic sieving phenomenon owes its name to the fact that,
mysteriously often, there is a particularly nice cyclic sieving
polynomial, for example a natural $q$-analogue of the counting
formula for the cardinality of $X$ as a function of $n$.  Moreover,
frequently the only known way to prove that a given $q$-analogue
indeed is a cyclic sieving polynomial is by enumerating the number of
fixed points of the group action and verifying that the evaluation of
the polynomial yields the same number.

\begin{remark}\label{rem:smallerGroupCSP}
  If $(X, \langle \rho \rangle, f(q))$ exhibits the cyclic sieving
  phenomenon, then so does $(X, \langle \rho^k \rangle, f(q))$ for
  any $k\in\setN$.  In this case also
  $(X^m, \langle \rho \rangle, f(q)^m)$ exhibits the cyclic sieving
  phenomenon for any $m\in\setN$, where $\langle \rho \rangle$ acts
  on $X^m$ via
  $\rho\cdot (x_1,\dotsc,x_m) = (\rho \cdot x_1, \dotsc, \rho \cdot
  x_m)$.
\end{remark}

Much attention has been given to prove cyclic sieving phenomena on certain sets of certain tableau objects.
Most famously, B.~Rhoades showed that $\SYT(a^b)$ exhibits the cyclic sieving phenomenon with the
group action of promotion and $f^\lambda(q)$ as cyclic sieving polynomial.
There are now several alternative proofs of this result, notably
\cite{Purbhoo2013,PetersenPylyavskyyRhoades2008} and \cite{ShenWeng2018}.
For an overview of some of these approaches, see \cite{Rhee2019}
\begin{table}[!ht]
	\centering
	\begin{tabular}{llllllc}
		\toprule
		Set &\phantom{}& Group action &\phantom{}&  Statistic/$f(q)$ &\phantom{}& Reference\\
		\midrule
		$\SYT(a^b)$ && Promotion && $\maj$ && \cite{Rhoades2010}\\
 		$\SYT((n-m,1^m))$ && Promotion\textsuperscript{\textdagger} && $\qbinom{n-1}{m}_q$ && \cite{BennettMadillStokke2014}\\
 		$\SYT(\lambda)$ && Evacuation\textsuperscript{\ddag} && $\maj$ && \cite{Stembridge1996}\\
		$\SSYT(a^b,k)$ && $k$-promotion && $q^{-\kappa(\lambda)}\schurS_{a^b}(1,q,\dotsc,q^{k-1})$ && \cite{Rhoades2010}\\
		$\COF(n \lambda/ n\mu)$ && Cyclic shift && (variant of) $\maj$ && \cite{AlexanderssonUhlin2019}\\
		\bottomrule
	\end{tabular}
	\caption{Summary of known cyclic sieving phenomena on tableau objects.
	\textsuperscript{\textdagger}Note that the group action on hook shaped SYT has order $n-1$.
	\textsuperscript{\ddag}Evacuation is an involution, so the cyclic group has order two.}

\end{table}

It turns out that it is possible to determine whether a polynomial
carries a cyclic group action of given order.
\begin{theorem}[{\cite[Thm. 2.7]{AlexanderssonAmini2018}}]\label{thm:alexanderssonAmini}
  Let $f(q)\in \setN[q]$ and suppose that $f(\xi^d)\in\setN$ for all
  $d\in\{1,\dots,n\}$, where $\xi$ is a primitive $n^\thsup$ root of
  unity.  Let $X$ be any set of size $f(1)$.  Then there exists a
  cyclic group action $\rho$ of order $n$ such that
  $(X,\langle\rho\rangle,f(q))$ exhibits the cyclic sieving phenomenon if and only if for
  every $k\mid n$,
  \[
  \sum_{d \mid k} \mu(k/d) f(\xi^d)\geq 0,
  \]
  where $\mu$ is the number-theoretic Möbius function.
\end{theorem}
\begin{remark}
  Except for its size, the nature of the set $X$ is irrelevant in
  this theorem.  Put differently, the theorem merely classifies the
  linear characters of the cyclic group of order $n$ which are the
  character of a group action.
\end{remark}
\begin{remark}\label{rmk:orbit-sizes}
  If $(X,\langle\rho\rangle,f(q))$ exhibits the cyclic sieving
  phenomenon, the expression
  \[\frac{1}{k}\sum_{d\mid k} \mu(k/d) f(\xi^d)\]
  is the number of orbits of size $k$ of the group action.
  Therefore, the sum
  \[
    \sum_{d\mid k} \mu(k/d) f(\xi^d)
  \]
  must be nonnegative and divisible by $k$.  The condition that the
  sum is divisible by $k$ follows from the hypothesis that
  $f(q)\in \setN[q]$ and $f(\xi^d)\in\setN$ for all
  $d\in\{1,\dots,n\}$, see~\cite[Lem.~2.5]{AlexanderssonAmini2018}.
\end{remark}
\begin{remark}
 It may be the case that $f(q)\in \setN[q]$ evaluates to nonnegative integers
 at $n^\thsup$ roots of unity, but is not a cyclic sieving polynomial.
 As an example (see~\cite[Ex.~2.10]{AlexanderssonAmini2018}), take $f(q)=q^5+3q^3+q+10$.
 At $6^\thsup$ roots of unity, $f(\xi^j)$ takes nonnegative integer values.
 However, for $k=3$ we have $\sum_{d\mid k} \mu(k/d) f(\xi^d) = -3$.
\end{remark}

We conclude this section by recalling a fact that makes cyclic groups
special.  In general, two non-isomorphic group actions may have the
same linear character. This is not the case for group actions of a
cyclic group, as R.~Brauer's permutation lemma shows:
\begin{theorem}[\cite{Brauer1941,Kovacs1982}]\label{thm:Brauer}
  Two cyclic group actions are isomorphic if and only if they are
  isomorphic as linear representation, that is, their characters coincide.
\end{theorem}

\section{Some properties of skew Schur functions}\label{sec:schur}

In this section we recall some basic properties of skew Schur
functions, the Littlewood--Richardson rule and fake degree polynomials.

Let $\SYT(\lambda/\mu)$ and $\SSYT(\lambda/\mu)$ denote the set of
standard and, respectively, semi-standard Young tableaux of
skew shape $\lambda/\mu$.  We refer to the books by I.~G.~Macdonald~\cite{Macdonald1995}
and R.~Stanley~\cite{StanleyEC2} for definitions.  We use English
notation in all our figures.

Given a skew shape $\lambda/\mu$ with $n$ boxes, the
associated \defin{skew Schur function} $\schurS_{\lambda/\mu}$ is defined as
\begin{equation}\label{eq:skewSchurDef}
 \schurS_{\lambda/\mu}(\xvec) \coloneqq
 \sum_{T \in \SSYT(\lambda/\mu)} \prod_{\square \in \lambda/\mu} x_{T(\square)}.
\end{equation}
This generalizes the ordinary \defin{Schur function} $s_\lambda \coloneqq s_{\lambda/\emptyset}$.
It is well-known that $\{ s_\lambda\}_\lambda$, where $\lambda$
runs over all partitions, is a basis for the ring of symmetric functions.
The \defin{power sum symmetric functions} indexed by partitions are defined as
\begin{equation}\label{eq:powerSumDef}
 \powerSum_\nu(\xvec) \coloneqq \powerSum_{\nu_1}(\xvec)\powerSum_{\nu_2}(\xvec)\dotsm \powerSum_{\nu_\ell}(\xvec),
 \quad \powerSum_{j}(\xvec)\coloneqq  x_1^j+x_2^j+\dotsb.
\end{equation}
The \defin{skew characters} $\chi^{\lambda/\mu}(\nu)$ of the symmetric
group $\symS_n$ are then defined via
\begin{equation}\label{eq:schurPExp}
 \schurS_{\lambda/\mu}(\xvec) = \sum_{\nu} \chi^{\lambda/\mu}(\nu)\frac{\powerSum_\nu(\xvec)}{z_\nu},
\end{equation}
where the sum is over all partitions $\nu$ of the same size as $\lambda/\mu$, $\powerSum_\nu$ denotes a \defin{power sum symmetric function}, $z_\nu =
\prod_j m_j! j^{m_j}$ and $m_j$ is the number of parts in $\lambda$ equal to $j$.

We shall also make use of the \defin{Littlewood--Richardson rule}, (see for example \cite{Macdonald1995,StanleyEC2}):
\begin{equation}\label{eq:lrRule}
 \schurS_{\lambda/\mu}(\xvec) = \sum_{\nu} c^{\lambda}_{\mu,\nu} \schurS_\nu(\xvec)
\end{equation}
where $c^{\lambda}_{\mu\nu} \in \setN$ are the \defin{Littlewood--Richardson coefficients}.
Combining \cref{eq:lrRule} and \cref{eq:schurPExp},
we obtain the Littlewood--Richardson rule for characters,
\begin{equation}\label{eq:characterLRRule}
\chi^{\lambda/\mu} = \sum_{\nu} c^{\lambda}_{\mu,\nu} \chi^{\nu}.
\end{equation}

Although it is very difficult to determine whether a given Littlewood--Richardson coefficient vanishes, the following particular case is straightforward.
\begin{lemma}\label{lemma:vanishing LR-coeffs}
Let $\lambda/\mu$ be a skew shape with $n$ boxes.
Then $c^\lambda_{\mu,(n)}=c^\lambda_{\mu,(1^n)}=0$ if and only if the skew shape $\lambda/\mu$
has some column with at least two boxes, and some row containing at least two boxes.
\end{lemma}
\begin{proof}
We shall first prove the
statement
\[
c^\lambda_{\mu,(n)}=0 \; \iff \; \lambda/\mu \text{ has some some column with at least two boxes}.
\]
We expand both sides of the Littlewood--Richardson rule~\eqref{eq:lrRule} in the monomial basis.
The left hand side contains the monomial $x_1^n$ if and only if
there is no column of $\lambda/\mu$ with at least two boxes.
Since the only semi-standard Young tableau of straight shape that contains precisely $n$ times the letter $1$ has shape $(n)$, the monomial $x_1^n$ appears in the right hand side if and only if $c^\lambda_{\mu,(n)}\neq 0$.

The statement concerning $c^\lambda_{\mu,(1^n)}$ follows by applying the involution\footnote{With the property that $\omega(s_{\lambda/\mu})=s_{{\lambda'}/{\mu'}}$, see \cite{Macdonald1995,StanleyEC2}.}
$\omega$ on both sides of the Littlewood--Richardson rule~\eqref{eq:lrRule}, which yields
\[
\schurS_{\lambda'/\mu'}(\xvec) = \sum_{\nu} c^{\lambda}_{\mu,\nu} \schurS_{\nu'}(\xvec).
\]
A similar argument as in the previous paragraph now finishes the proof.

\end{proof}

Alternatively, \cref{lemma:vanishing LR-coeffs} can also be proved using the fact that
$c^{\lambda}_{\mu,\nu}$ counts the number of standard Young tableaux of
shape $\lambda/\mu$ which are jeu-de-taquin equivalent to some fixed
standard Young tableau of shape $\nu$.

Note that jeu-de-taquin does not decrease the number of cells in a row,
nor the number of cells in a column.  Thus, if $\nu=(n)$ but
$\lambda/\mu$ has two cells in a column, the Littlewood--Richardson
coefficient $c^{\lambda}_{\mu,\nu}$ must vanish.  Conversely, filling a
shape $\lambda/\mu$ which does not have any two cells in the same column
from left to right with the numbers $1$ to $n$ yields a standard Young
tableau which is jeu-de-taquin equivalent to the unique tableau of shape
$(n)$.

Evidently, the same argument works for $\nu=(1^n)$.

\medskip

\begin{definition}\label{def:majAndfakeDegreePoly}

Given a skew standard Young tableau $T$ with $n$ boxes,
we say that $j \in [n-1]$ is a \defin{descent} of $T$
if $j+1$ appear in a row with index strictly higher than that of $j$.
We let the \defin{major index} of $T$, denoted $\maj(T)$, be the sum of the descents of $T$.
The \defin{fake-degree polynomial} $f^{\lambda/\mu}(q)$
associated with a skew Young diagram $\lambda/\mu$
is defined as the sum
\begin{equation}
 f^{\lambda/\mu}(q) \coloneqq \sum_{T \in \SYT(\lambda/\mu)} q^{\maj(T)}.
\end{equation}
\end{definition}

The following lemma relates the skew Schur functions and the fake-degree polynomials.
\begin{lemma}[{\cite[Prop. 7.19.11 ]{StanleyEC2}}] \label{lem:principalSpecSchur}
Let $\lambda/\mu$ be a skew shape with $n$ boxes. Then
 \[
  \schurS_{\lambda/\mu}(1,q,q^2,\dotsc) =
  \frac{ f^{\lambda/\mu}(q) }{ (1-q)(1-q^2)\dotsm (1-q^n) }.
 \]
\end{lemma}

\section{Skew characters and their fake degrees}\label{sec:skewCharacters}

A result by T.~Springer~\cite{Springer1974}
gives an expression for the evaluation of an irreducible character
of the symmetric group at a power of the \defin{long cycle} $\longcycle \in \symS_n$.
In this section, we extend this result to skew characters.
\begin{proposition}\label{thm:rootEvaluationGivesCharacter}
  Let $\lambda/\mu$ be a skew shape with $n$ boxes and let $\xi$ be a primitive
  $n^\thsup$ root of unity.  Then, for $n=dm$,
  \begin{align}\label{eq:rootEvaluationGivesCharacter}
    \chi^{\lambda/\mu}((m^d))  = f^{\lambda/\mu}(\xi^d).
  \end{align}
\end{proposition}

We shall first prove a more general result, from which we deduce \eqref{eq:rootEvaluationGivesCharacter}.
The following proposition was first proved explicitly by B.~Sagan, J.~Shareshian and M.~Wachs~\cite[Prop. 3.1]{SaganShareshianWachs2011}.
However, as they note, it is already implicit in work of J.~D{\'{e}}sarm{\'{e}}nien~\cite{Desarmenien1983}.
We include yet another, different, proof.

\begin{proposition}\label{prop:characterValuesFromPrincipalSpec}
 Let $F(\xvec)$ be a homogeneous symmetric function of degree $n$, such that
 \[
  F(\xvec) = \sum_{\nu \vdash n} \chi^F(\nu) \frac{\powerSum_\nu(\xvec)}{z_\nu}.
 \]
Let us furthermore introduce  $f^F(q)$ as
 \[
  f^F(q) \coloneqq \prod_{j=1}^n (1-q^j) F(1,q,q^2,\dotsc).
 \]
 Suppose that $f^F(q) \in \setZ[q]$.  Then, for a primitive
 $n^\thsup$ root of unity $\xi$ and $n=dm$, we have
 $\chi^F((m^d)) = f^F(\xi^d)$.
\end{proposition}
\begin{proof}
We have that $\powerSum_{k}(1,q,q^2,\dotsc) = (1-q^k)^{-1}$.
By definition,
 \[
f^F(q)= \left( \prod_{j=1}^n (1-q^j)\right)
\sum_{\nu \vdash n} \frac{\chi^{F}(\nu)}{z_\nu}\prod_{k=1}^{\ell(\nu)}\frac{1}{1-q^{\nu_k}}.
\]
Each summand on the right hand side approaches $0$ as $q\to \xi^d$ unless $\nu=(m^d)$,
as the first product has a zero with multiplicity $d$ at $q  = \xi^d$.
By only considering the terms involving $\nu=(m^d)$ in
the sum and rearranging the expression slightly, we have
\[
\lim_{q\to \xi^d} f^F(q) =
\lim_{q\to \xi^d}
\left( \prod_{k=1}^{d}\frac{1}{1-q^{km}} \prod_{j=1}^n (1-q^j)\right)
\frac{\chi^{F}( (m^d) )}{z_{(m^d)}}\prod_{k=1}^{d}\frac{1-q^{k m}}{1-q^{m}}.
\]
The last product approaches $d!$ due to l'Hospital's rule. After taking the limit and using $z_{(m^d)} = d! m^d$, we obtain
\[
f^F( \xi^d ) =
\left( \frac{1}{m} \prod_{j=1}^{m-1} (1-\xi^{jd})\right)^d \chi^{F}((m^d)).
\]
From the fact that $x^n-1 = \prod_{j=0}^{n-1}(x-\xi^j)$ we deduce
that $\prod_{j=1}^{m-1} (1-\xi^{jd}) = m$.
This shows that $f^F( \xi^d ) =  \chi^{F}((m^d))$, which proves our claim.
\end{proof}

\begin{proof}[Proof of \cref{thm:rootEvaluationGivesCharacter}]
  By \cref{lem:principalSpecSchur}, this is the special case of
  \cref{prop:characterValuesFromPrincipalSpec} with
  $F(\xvec) = \schurS_{\lambda/\mu}(\xvec)$.
\end{proof}

\section{Skew border-strip tableaux and the abacus}\label{sec:abaci}
In this section, we recall the definition of border-strip tableaux,
abaci and the Littlewood map. We end with a generalization of a
theorem of G.~James and A.~Kerber.

A \defin{border-strip} (or \defin{ribbon} or \defin{skew hook})
is a connected non-empty skew Young diagram containing no $2 \times 2$-square of boxes,
as in \cref{fig:borderStripExample}.
The \defin{height} $\height(B)$ of a border-strip $B$ is one less than the number of rows it spans.
\begin{figure}[!ht]
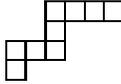

\ytableausetup{boxsize=0.7em}
    \[
    \ydiagram{2+4,2+1,3,1}
    \]
    \caption{A border-strip of height $3$.}
    \label{fig:borderStripExample}
\end{figure}

Let $\lambda/\mu$ be a skew shape. The \defin{size} of $\lambda/\mu$ is its number of boxes, denoted $|\lambda/\mu|$. Suppose that $\nu = (\nu_1, \dotsc, \nu_\ell)$ is a partition of $|\lambda/\mu|$. A \defin{border-strip tableau} of shape $\lambda/\mu$ and \defin{type} $\nu$
is a partition $B$ of the Young diagram of $\lambda/\mu$,
into labeled border-strips $B_1,\dotsc,B_\ell$
with the following properties:
\begin{itemize}
 \item the border-strip $B_j$ has label $j$ and size $\nu_j$,
 \item labeling all boxes in $B_j$ with $j$ results in a labeling of the diagram
 $\lambda/\mu$ where labels in every row and every column are weakly increasing.
\end{itemize}
We let $\BST(\lambda/\mu,\nu)$ denote the set of all such border-strip tableaux.
In particular, $\BST(\lambda/\mu,1^{n})$ may be identified with
the set of standard Young tableaux of shape $\lambda/\mu$.
In the remainder of the paper, we shall only concern ourselves with
border-strip tableaux where all strips have the same size.
We let \defin{$\BST(\lambda/\mu,d)$} denote the set of border-strip tableaux
where every strip have size $d$.

The \defin{height} of a border-strip tableau $T$, or any tiling of a tableau with border-strips, is the sum of the heights of the border-strips in
the partition.  The \defin{content} of a
box is given by its column index minus its row index.  Observe that
in a border-strip, the lowest leftmost box has the smallest content.
As a convention, the label of a strip is placed in the unique box with minimal
content in the strip, as done in \cref{fig:BSTExample}.

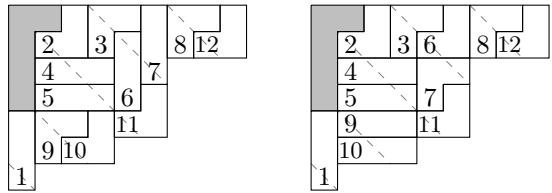
\begin{figure}[!ht]
\begin{tikzpicture}[baseline=(current bounding box.center)]
\begin{scope}[yscale=-1]
\draw[gray, very thin, dashed,x=1em,y=1em] (0,6) -- (1,7);
\draw[gray, very thin, dashed,x=1em,y=1em] (0,3) -- (3,6);
\draw[gray, very thin, dashed,x=1em,y=1em] (0,0) -- (5,5);
\draw[gray, very thin, dashed,x=1em,y=1em] (3,0) -- (6,3);
\draw[gray, very thin, dashed,x=1em,y=1em] (6,0) -- (8,2);
\draw[black,x=1em,y=1em,fill=lightgray] (0,0)--(2,0)--(2,1)--(1,1)--(1,4)--(0,4)--(0,0);
\begin{scope}[yshift=4em]
\pic{iTrio};
\end{scope}
\begin{scope}[xshift=1em,yshift=0em]
\pic{jTrio};
\end{scope}
\begin{scope}[xshift=3em,yshift=0em]
\pic{tTrio};
\end{scope}
\begin{scope}[xshift=1em,yshift=2em]
\pic{hTrio};
\end{scope}
\begin{scope}[xshift=1em,yshift=3em]
\pic{hTrio};
\end{scope}
\begin{scope}[xshift=4em,yshift=1em]
\pic{iTrio};
\end{scope}
\begin{scope}[xshift=5em,yshift=0em]
\pic{iTrio};
\end{scope}
\begin{scope}[xshift=6em,yshift=0em]
\pic{tTrio};
\end{scope}
\begin{scope}[xshift=1em,yshift=4em]
\pic{tTrio};
\end{scope}
\begin{scope}[xshift=2em,yshift=4em]
\pic{jTrio};
\end{scope}
\begin{scope}[xshift=4em,yshift=3em]
\pic{jTrio};
\end{scope}
\begin{scope}[xshift=7em,yshift=0em]
\pic{jTrio};
\end{scope}
\begin{scope}[xshift=-0.5em,yshift=-0.5em]
\node[x=1em,y=1em] at (1,4) {};
\node[x=1em,y=1em] at (1,7) {1};
\node[x=1em,y=1em] at (2,2) {2};
\node[x=1em,y=1em] at (4,2) {3};
\node[x=1em,y=1em] at (2,3) {4};
\node[x=1em,y=1em] at (2,4) {5};
\node[x=1em,y=1em] at (5,4) {6};
\node[x=1em,y=1em] at (6,3) {7};
\node[x=1em,y=1em] at (7,2) {8};
\node[x=1em,y=1em] at (2,6) {9};
\node[x=1em,y=1em] at (3,6) {\small{10}};
\node[x=1em,y=1em] at (5,5) {\small{11}};
\node[x=1em,y=1em] at (8,2) {\small{12}};
\end{scope}
\end{scope}
\end{tikzpicture}
\qquad
\begin{tikzpicture}[baseline=(current bounding box.center)]
\begin{scope}[yscale=-1]
\draw[gray, very thin, dashed,x=1em,y=1em] (0,6) -- (1,7);
\draw[gray, very thin, dashed,x=1em,y=1em] (0,3) -- (3,6);
\draw[gray, very thin, dashed,x=1em,y=1em] (0,0) -- (5,5);
\draw[gray, very thin, dashed,x=1em,y=1em] (3,0) -- (6,3);
\draw[gray, very thin, dashed,x=1em,y=1em] (6,0) -- (8,2);
\draw[black,x=1em,y=1em,fill=lightgray] (0,0)--(2,0)--(2,1)--(1,1)--(1,4)--(0,4)--(0,0);
\begin{scope}[yshift=4em]
\pic{iTrio};
\end{scope}
\begin{scope}[xshift=1em,yshift=0em]
\pic{jTrio};
\end{scope}
\begin{scope}[xshift=3em,yshift=0em]
\pic{tTrio};
\end{scope}
\begin{scope}[xshift=1em,yshift=2em]
\pic{hTrio};
\end{scope}
\begin{scope}[xshift=1em,yshift=3em]
\pic{hTrio};
\end{scope}
\begin{scope}[xshift=4em,yshift=2em]
\pic{tTrio};
\end{scope}
\begin{scope}[xshift=4em,yshift=0em]
\pic{jTrio};
\end{scope}
\begin{scope}[xshift=6em,yshift=0em]
\pic{tTrio};
\end{scope}
\begin{scope}[xshift=1em,yshift=4em]
\pic{hTrio};
\end{scope}
\begin{scope}[xshift=1em,yshift=5em]
\pic{hTrio};
\end{scope}
\begin{scope}[xshift=4em,yshift=3em]
\pic{jTrio};
\end{scope}
\begin{scope}[xshift=7em,yshift=0em]
\pic{jTrio};
\end{scope}
\begin{scope}[xshift=-0.5em,yshift=-0.5em]
\node[x=1em,y=1em] at (1,4) {};
\node[x=1em,y=1em] at (1,7) {1};
\node[x=1em,y=1em] at (2,2) {2};
\node[x=1em,y=1em] at (4,2) {3};
\node[x=1em,y=1em] at (2,3) {4};
\node[x=1em,y=1em] at (2,4) {5};
\node[x=1em,y=1em] at (5,2) {6};
\node[x=1em,y=1em] at (5,4) {7};
\node[x=1em,y=1em] at (7,2) {8};
\node[x=1em,y=1em] at (2,5) {9};
\node[x=1em,y=1em] at (2,6) {\small{10}};
\node[x=1em,y=1em] at (5,5) {\small{11}};
\node[x=1em,y=1em] at (8,2) {\small{12}};
\end{scope}
\end{scope}
\end{tikzpicture}
\caption{Two border-strip tableaux in $\BST((9^2,\!6^3,\!4,\!1)/(2,\!1^3), 3)$.
In each strip, the label has been placed in the box with minimal content.
}
\label{fig:BSTExample}
\end{figure}

In \cref{eq:schurPExp}, the skew characters $\chi^{\lambda/\mu}$ were defined.
The skew Murnaghan--Nakayama rule describes a way to compute these skew characters.
\begin{theorem}[{Murnaghan--Nakayama, see \cite[Cor. 7.17.5 ]{StanleyEC2}}]\label{thm:mnRule}
The skew characters are given by the signed sum
\begin{equation}\label{eq:mnRule}
\chi^{\lambda/\mu}(\nu) = \sum_{B \in \BST(\lambda/\mu,\nu)} (-1)^{\height(B)}.
\end{equation}
\end{theorem}

\begin{definition}
  Given a partition $\lambda$ with $\ell$ parts and an integer
  $d\geq 1$, we define an \defin{abacus} with $d$ \defin{runners} as
  follows.

  Let $b_1<b_2<\dots<b_\ell$ be the hook lengths of the boxes in the
  first column of $\lambda$, see \cref{fig:binaryWord}. Then, proceeding row by row from left
  to right, label the positions on the runners of the abacus from $0$
  to $b_\ell$, placing a bead at each of the positions labeled
  $b_1,b_2,\dots,b_\ell$.
\end{definition}

We can now define the core and the quotient of a partition.
\begin{definition}
  Let $\lambda$ be a partition and let $d\geq 1$ be an integer.  Then
  the \defin{$d$-core} of $\lambda$ is the partition corresponding to
  the $d$-abacus with all beads moved up as far as possible.  The
  \defin{$d$-quotient} of $\lambda$ is the $d$-tuple of partitions
  obtained by regarding each runner as an individual $1$-abacus.
\end{definition}

Next, we extend these definitions to skew shapes. To do so,
we need the following observation.
\begin{proposition}
  Let $\lambda$ be a partition.  Suppose that $\mu$ is obtained from
  $\lambda$ by removing a border-strip of size $d$.  Then the
  $d$-abacus corresponding to $\mu$ is obtained from the $d$-abacus
  corresponding to $\lambda$ by moving a bead up by one position
  along its runner.
\end{proposition}
Therefore, if $\lambda/\mu$ is a skew shape such that
$\BST(\lambda/\mu, d)$ is non-empty, the abacus corresponding to
$\lambda$ can be obtained from the abacus corresponding to $\mu$ by
moving down beads along their runners.  Note that in this case the
$d$-cores of $\lambda$ and $\mu$ coincide.
\begin{definition}
  The \defin{$d$-quotient} of a skew shape $\lambda/\mu$ is the
  $d$-tuple of skew shapes obtained by regarding each
  corresponding pair of runners in the pair of $d$-abaci for
  $\lambda$ and $\mu$ as a pair of $1$-abaci.
\end{definition}

\begin{note}
  We use the slightly nonstandard convention fixing the number of
  beads to be equal to the number of parts $\ell(\lambda)$ of $\lambda$.
\end{note}

Finally, we recall the \defin{Littlewood map} as described by, for example,
G.~James \& A.~Kerber~\cite[Ch.~2.7]{JamesKerber1984} or I.~Pak~\cite[fig.~2.6]{Pak2000Ribbon}, where it is called the rim hook bijection.

Let \defin{$\SYTT(\lambda/\mu, d)$} be the set of $d$-tuples
$(T^1,T^2,\dotsc,T^d)$ of tableaux with the following properties:
\begin{itemize}
 \item the shape of $T^j$ is $\lambda^j/\mu^j$, the $j^\thsup$ entry of the $d$-quotient of $\lambda/\mu$,
 \item for each tableau $T^j$, the box labels in rows and columns increase,
 \item each of the numbers $\{1,2,\dotsc,n/d\}$ appears in precisely one tableau.
\end{itemize}

Then the Littlewood map is the following bijection between
$\BST(\lambda/\mu,d)$ and $\SYTT(\lambda/\mu, d)$.  Let
$B \in \BST(\lambda/\mu,d)$.  Let $c(x)$ denote the content of the unique box with minimal
content in the $d$-strip labeled $x$, as shown in \cref{fig:BSTExample}.
Then the boxes in $T^j$ are filled with
the labels $x$ such that $c(x) \equiv j-\ell(\lambda) \mod d$.  Furthermore,
the relative position of the labels in $T^j$ is the same as in the border-strip tableau.  In particular,
two labeled boxes in $T^j$ differ in content by $k$ if and only if the
corresponding labels differ in content by $dk$ in $B$.

The image of the two border-strip tableaux in \cref{fig:BSTExample}
are given by the triples in \eqref{eq:bstLittlewoodImage1} and
\eqref{eq:bstLittlewoodImage2}, respectively.

\begin{example}
The $3$-cores of $\lambda = (9^2,\!6^3,\!4,\!1)$ and $\mu=(2,\!1^3)$ are both given by the
partition $(2)$.
The $3$-quotient of $\lambda$ is $(4,\!3), (2), (2,\!1^2)$
and the $3$-quotient of $\mu$ is $(1), \emptyset, \emptyset$.
Thus, the $3$-quotient of $\lambda/\mu$ is $(4,\!3)/(1),\; (2),\; (2,\!1^2)$
and the two tuples
\begin{align}
\ytableausetup{boxsize=1em}
&\ytableaushort{{\none}27{12},1{10}{11}}\quad  \ytableaushort{56}\quad  \ytableaushort{38,4,{9}} \label{eq:bstLittlewoodImage1} \\
&\ytableaushort{{\none}26{12},1{9}{11}}\quad  \ytableaushort{57}\quad  \ytableaushort{38,4,{10}} \label{eq:bstLittlewoodImage2}
\end{align}
are elements in $\SYTT(\lambda/\mu,3)$, corresponding to the two tableaux in \cref{fig:BSTExample}.
\end{example}

\begin{figure}[!ht]
\begin{tikzpicture}[baseline=(current bounding box.center)]
\begin{scope}[yscale=-1]

\draw[gray, very thin, dashed,x=1em,y=1em] (0,6) -- (1,7);
\draw[gray, very thin, dashed,x=1em,y=1em] (0,3) -- (3,6);
\draw[gray, very thin, dashed,x=1em,y=1em] (0,0) -- (5,5);
\draw[gray, very thin, dashed,x=1em,y=1em] (3,0) -- (6,3);
\draw[gray, very thin, dashed,x=1em,y=1em] (6,0) -- (8,2);

\draw[black,x=1em,y=1em,fill=lightgray] (0,0)--(2,0)--(2,1)--(1,1)--(1,4)--(0,4)--(0,0);
\begin{scope}[yshift=4em]
\pic{iTrio};
\end{scope}
\begin{scope}[xshift=1em,yshift=0em]
\pic{jTrio};
\end{scope}
\begin{scope}[xshift=3em,yshift=0em]
\pic{tTrio};
\end{scope}
\begin{scope}[xshift=1em,yshift=2em]
\pic{hTrio};
\end{scope}
\begin{scope}[xshift=1em,yshift=3em]
\pic{hTrio};
\end{scope}
\begin{scope}[xshift=4em,yshift=1em]
\pic{iTrio};
\end{scope}
\begin{scope}[xshift=5em,yshift=0em]
\pic{iTrio};
\end{scope}
\begin{scope}[xshift=6em,yshift=0em]
\pic{tTrio};
\end{scope}
\begin{scope}[xshift=1em,yshift=4em]
\pic{tTrio};
\end{scope}
\begin{scope}[xshift=2em,yshift=4em]
\pic{jTrio};
\end{scope}
\begin{scope}[xshift=4em,yshift=3em]
\pic{jTrio};
\end{scope}
\begin{scope}[xshift=7em,yshift=0em]
\pic{jTrio};
\end{scope}

\begin{scope}[xshift=-0.5em,yshift=-0.5em]
\node[x=1em,y=1em] at (-1,7) {1};
\node[x=1em,y=1em] at (-1,6) {5};
\node[x=1em,y=1em] at (-1,5) {8};
\node[x=1em,y=1em] at (-1,4) {9};
\node[x=1em,y=1em] at (-1,3) {10};
\node[x=1em,y=1em] at (-1,2) {14};
\node[x=1em,y=1em] at (-1,1) {15};
\node[x=1em,y=1em] at (-2.5,4) {1};
\node[x=1em,y=1em] at (-2.5,3) {2};
\node[x=1em,y=1em] at (-2.5,2) {3};
\node[x=1em,y=1em] at (-2.5,1) {5};
\end{scope}
\end{scope}
\end{tikzpicture}
  \qquad
    \begin{tikzpicture}[baseline=2.2cm, xscale=2, yscale=0.5, %
      bead/.style={shape=circle, draw, fill=white, minimum size=8pt, inner sep = 0pt,%
execute at begin node=$\scriptstyle,%
execute at end node=$%
      }, %
      nonbead/.style={shape=circle, fill=white,inner sep=0.2pt,%
execute at begin node=$\scriptstyle,%
execute at end node=$},%
      runnerlabel/.style={shape=circle,inner sep=0.2pt,%
execute at begin node=$\scriptstyle,%
execute at end node=$%
      },
      bullet/.style={shape=circle, fill=black, minimum size=8pt}]
      \draw (-0.25, 9.5) -- (2.25, 9.5);
      \draw (0, 7) -- (0, 9.5);
      \draw (1, 7) -- (1, 9.5);
      \draw (2, 7) -- (2, 9.5);
      \node[bullet] at (0, 9) {};
      \node[nonbead] at (0, 8) {0};
      \node[bead] at (0, 7) {3};
      \node[bullet] at (1, 9) {};
      \node[bead] at (1, 8) {1};
      \node[nonbead] at (1, 7) {4};
      \node[bullet] at (2, 9) {};
      \node[bead] at (2, 8) {2};
      \node[bead] at (2, 7) {5};
      \draw (-0.25, 6.5) -- (2.25, 6.5);
      \draw (0, 1) -- (0, 6.5);
      \draw (1, 1) -- (1, 6.5);
      \draw (2, 1) -- (2, 6.5);
      \node[nonbead] at (0, 6) {0};
      \node[nonbead] at (0, 5) {3};
      \node[nonbead] at (0, 4) {6};
      \node[bead] at (0, 3) {9};
      \node[nonbead] at (0, 2) {12};
      \node[bead] at (0, 1) {15};
      \node[runnerlabel] at (0, 0) {\nu^1/\kappa^1=(4,3)/(1)};
      \node[bead] at (1, 6) {1};
      \node[nonbead] at (1, 5) {4};
      \node[nonbead] at (1, 4) {7};
      \node[bead] at (1, 3) {10};
      \node[nonbead] at (1, 2) {13};
      \node[runnerlabel] at (1, 0) {\nu^2/\kappa^2=\emptyset};
      \node[nonbead] at (2, 6) {2};
      \node[bead] at (2, 5) {5};
      \node[bead] at (2, 4) {8};
      \node[nonbead] at (2, 3) {11};
      \node[bead] at (2, 2) {14};
      \node[runnerlabel] at (2, 0) {\nu^3/\kappa^3=(2,1^2)};
    \end{tikzpicture}
\caption{The hook lengths of the first column of the inner and the outer shape of the diagram in \cref{fig:BSTExample}, and the pair of abaci corresponding to the shape.
}
\label{fig:binaryWord}
\end{figure}
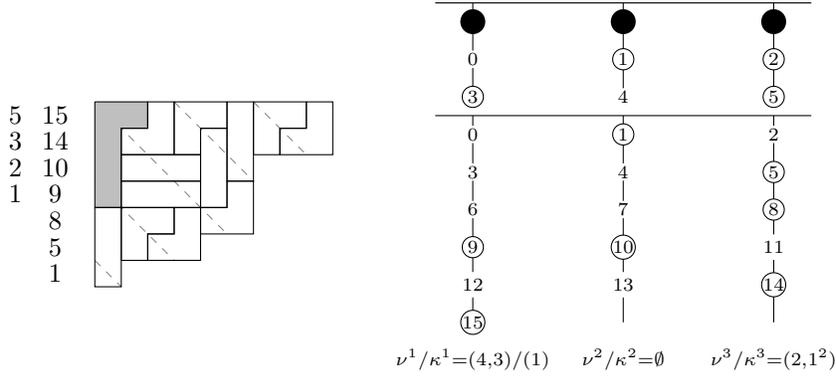

\begin{lemma}\label{lem:stripHook}
  Let $\lambda/\mu$ be a skew shape of size $n$ and let $k$ and $d$
  be positive integers with $dk \mid n$.  Suppose that
  $\BST(\lambda/\mu, dk)$ is non-empty and let
  $(\nu^1/\kappa^1,\dotsc,\nu^k/\kappa^k)$ be the skew $k$-quotient
  of $\lambda/\mu$.  Then
  $d\mid\gcd(|\nu^1/\kappa^1|,\dotsc,|\nu^k/\kappa^k|)$ and
  \begin{equation}\label{eq:stripHook}
    |\BST(\lambda/\mu, dk)| = %
    \binom{\sum_i|\nu^i/\kappa^i|/d} %
    {|\nu^1/\kappa^1|/d,\dotsc,|\nu^k/\kappa^k|/d} %
    \prod_{i=1}^k |\BST(\nu^i/\kappa^i, d)|.
 \end{equation}
\end{lemma}
\begin{remark}
  The case $d=1$ for straight shapes is classical and can be found
  in~\cite[eq.~(2.7.32)]{JamesKerber1984} or \cite{FominLulov1997}.
  In this special case, the result also follows immediately from the Littlewood map.
\end{remark}
\begin{proof}
  We reduce the statement to the case $d=1$.  Let
  $\nu^{i,1}/\kappa^{i,1},\dots,\nu^{i,d}/\kappa^{i,d}$ be the
  $d$-quotient of $\nu^i/\kappa^i$, for $1\leq i\leq k$.  We first
  show that the collection of skew shapes
  $\nu^{i,j}/\kappa^{i,j}$ for $1\leq i\leq k$ and $1\leq j\leq d$
  coincides with the collection of skew shapes in the
  $dk$-quotient for $\lambda/\mu$.

  To do so, we show that the bead positions on the
  $\big(i+(j-1)k\big)^\thsup$ runner of the $dk$-abacus for
  $\lambda/\mu$ coincide with the bead positions on the $j^\thsup$
  runner of the $d$-abacus for the skew shape corresponding to the
  $i^\thsup$ runner of the $k$-abacus for $\lambda/\mu$, for all
  $1\leq i\leq k$ and $1\leq j\leq d$.

  Suppose that the labeled beads on the abacus for $\lambda/\mu$ are
  $b_1 < b_2 < \dots < b_\ell$.  Recall our convention that the
  number of beads on the abacus equals the number of rows of the
  outer partition.  Thus, the \emph{positions} ($0$ being topmost) of
  the beads on the $i^\thsup$ runner of the $k$-abacus for
  $\lambda/\mu$ are
  $\left\{\frac{b_m-(i-1)}{k}: b_m = i-1\pmod{k}\right\}$.

  Then, on the one hand, the $j^\thsup$ runner of the $d$-abacus for
  the skew shape corresponding to that runner has beads at
  positions
  \[
  \left\{\frac{\frac{b_m-(i-1)}{k}-(j-1)}{d}: b_m = i-1\!\pmod{k}
    \text{ and } \frac{b_m-(i-1)}{k}=j-1\!\pmod{d}\right\}.
  \]
  On the other hand, the $\big(i+(j-1)k\big)^\thsup$ runner of the
  $dk$-abacus for $\lambda/\mu$ has beads at positions
  $\left\{\frac{b_m-(i+(j-1)k-1)}{dk}: b_m =
    i+(j-1)k-1\pmod{dk}\right\}$, which are indeed the same.

  To conclude the argument, we compute
  \begin{align*}
    &\binom{|\lambda/\mu|/dk} %
    {|\nu^1/\kappa^1|/d,\dotsc,|\nu^k/\kappa^k|/d} %
    \prod_{i=1}^k |\BST(\nu^i/\kappa^i, d)|\\
    &= \binom{|\lambda/\mu|/dk} %
      {|\nu^1/\kappa^1|/d,\dotsc,|\nu^k/\kappa^k|/d} %
      \prod_{i=1}^k \binom{|\nu^i/\kappa^i|/d} %
      {|\nu^{i,1}/\kappa^{i,1}|,\dotsc,|\nu^{i,d}/\kappa^{i,d}|} %
      \prod_{j=1}^d |\SYT(\nu^{i,j}/\kappa^{i,j})|\\
    &= \binom{|\lambda/\mu|/dk} %
      {|\nu^{1,1}/\kappa^{1,1}|,\dotsc,|\nu^{k,d}/\kappa^{k,d}|} %
      \prod_{i=1}^k \prod_{j=1}^d |\SYT(\nu^{i,j}/\kappa^{i,j})|\\
    &=|\BST(\lambda/\mu, dk)|.
  \end{align*}
\end{proof}

\begin{example}
  Consider the skew shape $\lambda/\mu=(9,\!7,\!4^2,\!3^2,\!1)/(3,\!2,\!1^2)$ of $24$.
  The corresponding pair of $6$-abaci is as follows.  Note that the
  lower abacus, corresponding to the outer shape $\lambda$, is obtained from
  the upper abacus, corresponding to the inner shape $\mu$, by moving down
  beads along a runner:
  
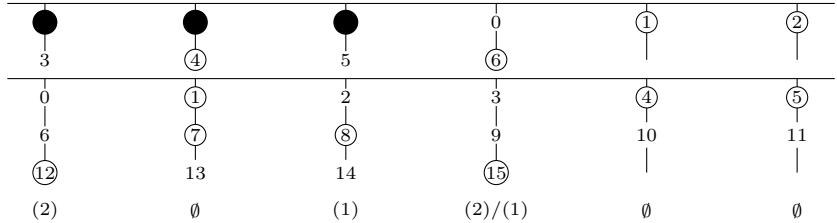
\[
    \begin{tikzpicture}[xscale=2, yscale=0.5, %
      bead/.style={shape=circle, draw, fill=white, minimum size=8pt, inner sep = 0pt,%
execute at begin node=$\scriptstyle,%
execute at end node=$%
      }, %
      nonbead/.style={shape=circle, fill=white, inner sep=0.2pt,%
execute at begin node=$\scriptstyle,%
execute at end node=$%
      },
      runnerlabel/.style={shape=circle,inner sep=0.2pt,%
execute at begin node=$\scriptstyle,%
execute at end node=$%
      },
      bullet/.style={shape=circle, fill=black, minimum size=8pt}]
      \draw (-0.25, 6.5) -- (5.25, 6.5);
      \draw (0, 5) -- (0, 6.5);
      \draw (1, 5) -- (1, 6.5);
      \draw (2, 5) -- (2, 6.5);
      \draw (3, 5) -- (3, 6.5);
      \draw (4, 5) -- (4, 6.5);
      \draw (5, 5) -- (5, 6.5);
      \node[bullet] at (0, 6) {};
      \node[nonbead] at (0, 5) {3};
      \node[bullet] at (1, 6) {};
      \node[bead] at (1, 5) {4};
      \node[bullet] at (2, 6) {};
      \node[nonbead] at (2, 5) {5};
      \node[nonbead] at (3, 6) {0};
      \node[bead] at (3, 5) {6};
      \node[bead] at (4, 6) {1};
      \node[bead] at (5, 6) {2};
      \draw (-0.25, 4.5) -- (5.25, 4.5);
      \draw (0, 2) -- (0, 4.5);
      \draw (1, 2) -- (1, 4.5);
      \draw (2, 2) -- (2, 4.5);
      \draw (3, 2) -- (3, 4.5);
      \draw (4, 2) -- (4, 4.5);
      \draw (5, 2) -- (5, 4.5);
      \node[nonbead] at (0, 4) {0};
      \node[nonbead] at (0, 3) {6};
      \node[bead] at (0, 2) {12};
      \node[runnerlabel] at (0, 1) {(2)};
      \node[bead] at (1, 4) {1};
      \node[bead] at (1, 3) {7};
      \node[nonbead] at (1, 2) {13};
      \node[runnerlabel] at (1, 1) {\emptyset};
      \node[nonbead] at (2, 4) {2};
      \node[bead] at (2, 3) {8};
      \node[nonbead] at (2, 2) {14};
      \node[runnerlabel] at (2, 1) {(1)};
      \node[nonbead] at (3, 4) {3};
      \node[nonbead] at (3, 3) {9};
      \node[bead] at (3, 2) {15};
      \node[runnerlabel] at (3, 1) {(2)/(1)};
      \node[bead] at (4, 4) {4};
      \node[nonbead] at (4, 3) {10};
      \node[runnerlabel] at (4, 1) {\emptyset};
      \node[bead] at (5, 4) {5};
      \node[nonbead] at (5, 3) {11};
      \node[runnerlabel] at (5, 1) {\emptyset};
    \end{tikzpicture}
  \]
  On the other hand, its pair of $3$-abaci is the following:
  
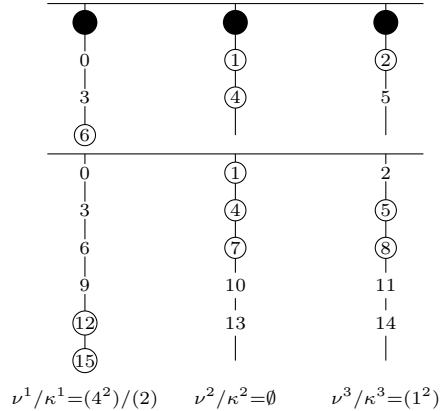
\[
    \begin{tikzpicture}[xscale=2, yscale=0.5, %
      bead/.style={shape=circle, draw, fill=white, minimum size=8pt, inner sep = 0pt,%
execute at begin node=$\scriptstyle,%
execute at end node=$%
      }, %
      nonbead/.style={shape=circle, fill=white,inner sep=0.2pt,%
execute at begin node=$\scriptstyle,%
execute at end node=$},%
      runnerlabel/.style={shape=circle,inner sep=0.2pt,%
execute at begin node=$\scriptstyle,%
execute at end node=$%
      },
      bullet/.style={shape=circle, fill=black, minimum size=8pt}]
      \draw (-0.25, 10.5) -- (2.25, 10.5);
      \draw (0, 7) -- (0, 10.5);
      \draw (1, 7) -- (1, 10.5);
      \draw (2, 7) -- (2, 10.5);
      \node[bullet] at (0, 10) {};
      \node[nonbead] at (0, 9) {0};
      \node[nonbead] at (0, 8) {3};
      \node[bead] at (0, 7) {6};
      \node[bullet] at (1, 10) {};
      \node[bead] at (1, 9) {1};
      \node[bead] at (1, 8) {4};
      \node[bullet] at (2, 10) {};
      \node[bead] at (2, 9) {2};
      \node[nonbead] at (2, 8) {5};
      \draw (-0.25, 6.5) -- (2.25, 6.5);
      \draw (0, 1) -- (0, 6.5);
      \draw (1, 1) -- (1, 6.5);
      \draw (2, 1) -- (2, 6.5);
      \node[nonbead] at (0, 6) {0};
      \node[nonbead] at (0, 5) {3};
      \node[nonbead] at (0, 4) {6};
      \node[nonbead] at (0, 3) {9};
      \node[bead] at (0, 2) {12};
      \node[bead] at (0, 1) {15};
      \node[runnerlabel] at (0, 0) {\nu^1/\kappa^1=(4^2)/(2)};
      \node[bead] at (1, 6) {1};
      \node[bead] at (1, 5) {4};
      \node[bead] at (1, 4) {7};
      \node[nonbead] at (1, 3) {10};
      \node[nonbead] at (1, 2) {13};
      \node[runnerlabel] at (1, 0) {\nu^2/\kappa^2=\emptyset};
      \node[nonbead] at (2, 6) {2};
      \node[bead] at (2, 5) {5};
      \node[bead] at (2, 4) {8};
      \node[nonbead] at (2, 3) {11};
      \node[nonbead] at (2, 2) {14};
      \node[runnerlabel] at (2, 0) {\nu^3/\kappa^3=(1^2)};
    \end{tikzpicture}
  \]
  Finally, as illustration of the proof, consider the pair of
  $2$-abaci for $\nu^1/\kappa^1$, corresponding to the first runner
  in the $3$-abacus above.  Observe that the \emph{positions} of the
  beads, and therefore also the corresponding skew shapes, are
  the same as on the first and the fourth runner of the $6$-abacus:
  
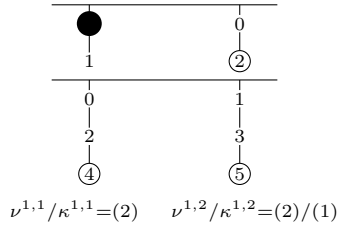
\[
    \begin{tikzpicture}[xscale=2, yscale=0.5, %
      bead/.style={shape=circle, draw, fill=white, minimum size=8pt, inner sep = 0pt,%
execute at begin node=$\scriptstyle,%
execute at end node=$%
      }, %
      nonbead/.style={shape=circle, fill=white,inner sep=0.2pt,%
execute at begin node=$\scriptstyle,%
execute at end node=$},%
      runnerlabel/.style={shape=circle,inner sep=0.2pt,%
execute at begin node=$\scriptstyle,%
execute at end node=$%
      },
      bullet/.style={shape=circle, fill=black, minimum size=8pt}]
      \draw (-0.25, 5.5) -- (1.25, 5.5);
      \draw (0, 4) -- (0, 5.5);
      \draw (1, 4) -- (1, 5.5);
      \node[bullet] at (0, 5) {};
      \node[nonbead] at (0, 4) {1};
      \node[nonbead] at (1, 5) {0};
      \node[bead] at (1, 4) {2};
      \draw (-0.25, 3.5) -- (1.25, 3.5);
      \draw (0, 1) -- (0, 3.5);
      \draw (1, 1) -- (1, 3.5);
      \node[nonbead] at (0, 3) {0};
      \node[nonbead] at (0, 2) {2};
      \node[bead] at (0, 1) {4};
      \node[runnerlabel] at (-0.1, 0) {\nu^{1,1}/\kappa^{1,1}=(2)};
      \node[nonbead] at (1, 3) {1};
      \node[nonbead] at (1, 2) {3};
      \node[bead] at (1, 1) {5};
      \node[runnerlabel] at (1.1, 0) {\nu^{1,2}/\kappa^{1,2}=(2)/(1)};
    \end{tikzpicture}
  \]
\end{example}

As a corollary we obtain a useful characterization of shapes with
precisely one border-strip tableau.
\begin{corollary}\label{lem:shapeAbacusPrimes}
  Let $\lambda/\mu$ be a skew shape of size $n$ and let $k$ be a
  positive integer with $k\mid n$.  Suppose that $\BST(\lambda/\mu, k)$
  is non-empty.

  Then $\BST(\lambda/\mu, k)$ contains precisely one element if and
  only if all skew shapes in the skew $k$-quotient of
  $\lambda/\mu$ are empty, with one one exception, which is either a
  single row $(n/k)$ or a single column $(1^{n/k})$.

  In this case $\lambda/\mu$ is a border-strip.  In particular,
  $\BST(\lambda/\mu, dk)$ contains precisely one element for all
  $d\mid\frac{n}{k}$.

\end{corollary}
\begin{proof}
  Let $(\nu^1/\kappa^1,\dotsc,\nu^k/\kappa^k)$ be the skew
  $k$-quotient of $\lambda/\mu$ and suppose that
  $\BST(\lambda/\mu, k)$ contains precisely one element.  Then,
  applying \cref{lem:stripHook} with $d=1$ we obtain that
  $|\BST(\lambda/\mu, k)|=1$ if and only if
  $|\BST(\nu^i/\kappa^i, 1)|=1$ for all $i$ and the multinomial
  coefficient in \cref{eq:stripHook} evaluates to $1$.

  Since $\BST(\nu^i/\kappa^i, 1)$ is the set of standard Young
  tableaux of shape $\nu^i/\kappa^i$, we have that
  $|\BST(\nu^i/\kappa^i, 1)|=1$ if and only $\nu^i/\kappa^i$ is a
  single row or a single column.  Furthermore, the multinomial
  coefficient equals $1$ if $|\nu^i/\kappa^i| = 0$ for all but one
  $i$.

  Finally, suppose that $\BST(\lambda/\mu, k)$ contains precisely one
  element.  It follows immediately that $\lambda/\mu$ is connected.
  Let $\nu^i/\kappa^i$ be the unique non-empty element in the
  quotient.  Then, by definition of the Littlewood map, the contents
  of the boxes with minimal content in $\lambda/\mu$ are all
  different, and all equivalent to $i\pmod k$.  Therefore,
  $\lambda/\mu$ must be a border-strip.

  This in turn implies that $\BST(\lambda/\mu, dk)$ is non-empty for
  all $d\mid\frac{n}{k}$, and therefore contains precisely one
  element.
\end{proof}

\section{Skew characters and border-strip tableaux}\label{sec:borderStrips}

In this section we show that, up to sign, the evaluation of a skew
character at a $d^\thsup$ power of a cycle equals the number of border
strip tableaux with all strips having size $d$. The non-skew case follows from a result by D.~White \cite{White1983}
and via a different technique by G.~James and A.~Kerber~\cite[Eq. 2.7.36]{JamesKerber1984}.
Our proof is slightly different and uses the techniques by I.~Pak~\cite{Pak2000Ribbon}.

\begin{definition}
  Let $T=(T^1,\dotsc,T^d) \in \SYTT(\lambda/\mu,d)$.  Suppose that,
  after swapping entries $i$ and $i+1$ in $T$, the resulting tuple is
  still an element of $\SYTT(\lambda/\mu,d)$.  We refer to such a
  transposition as a \defin{flip} on $T$.
\end{definition}

\begin{example}
	The two flips $(6,7)$ and $(9,10)$ send the tableau \eqref{eq:bstLittlewoodImage1} to the tableau \eqref{eq:bstLittlewoodImage2}.
\end{example}

\begin{lemma}\label{lem:flipConnected}
	All elements of $\SYTT(\lambda/\mu,d)$ are connected via a sequence of flips.
\end{lemma}
This lemma is essentially \cite[Thm. 3.2]{Pak2000Ribbon}, although it is stated in a slightly different way.
We therefore include a proof using the framework in this paper.
\begin{proof}
Let $(\lambda^1/\mu^1,\dotsc,\lambda^d/\mu^d)$ be the $d$-quotient of $\lambda/\mu$.
Let us first describe the \defin{superstandard filling}
$S \coloneqq (S^1,\dotsc, S^d)\in \SYTT(\lambda/\mu,d)$.  We will then show that $S$ can be obtained from any other tableau by a sequence of flips.

The boxes in $S^1$ are labeled with the numbers $1,\dotsc,|\lambda^1/\mu^1|$,
the boxes in $S^2$ are labeled with numbers $|\lambda^1/\mu^1|+1, \dotsc, |\lambda^1/\mu^1|+|\lambda^2/\mu^2|$, and so forth.
The entries in each tableau $S^i$ are then distributed in the lexicographically smallest fashion,
when reading row by row from top to bottom.

It now suffices to prove that for arbitrary $T \in \SYTT(\lambda/\mu,d)$, we can obtain $S$ from $T$ by a
sequence of flips. We describe a sorting algorithm which rearranges the labels,
starting with $1$ and then continuing with $2,3,\dotsc$ so that these labels agree with the corresponding entries in $S$.

Suppose that, at some point during the procedure, all boxes in $T$ with labels at most $i-1$
agree with $S$, but the box labeled $i$ in $S$ is labeled $j>i$ in $T$.
We then first flip $j$ and $j-1$, then $j-1$ and $j-2$, etc.\ until $i+1$ and $i$ are flipped, at which point $i$ is in the correct spot.
We show inductively that all these flips are possible to perform: by construction,
the boxes above and to the left of the box labeled with $j$ in $T$ must have labels smaller than $i < j-1$.
Because all tableaux in $T$ are standard, the boxes below and to the right of the box labeled with $j-1$ contain labels strictly larger than $j$.
Performing this algorithm in order for all $i=1,2,\dotsc$ ensures that we eventually reach $S$ from $T$.
\end{proof}

We shall now discuss the action of flipping $i$ and $i+1$ on a border-strip tableau.
By definition, and since the Littlewood map is bijective, the image of a flip also corresponds to a valid border-strip tableau.
There are two possible cases. Either, the strips labeled $i$ and $i+1$
are disconnected (no box in the first strip is horizontally or vertically adjacent to a box in the second),
in which case the labels are just interchanged.
The second case --- when the strips are connected as in \eqref{eq:flip} --- is more interesting.
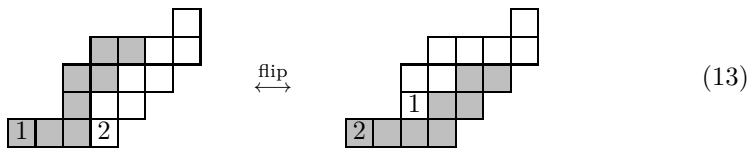
\begin{align}\label{eq:flip}
\begin{ytableau}
\none & \none & \none & \none & \none & \none & *(white)\\
\none & \none & \none & *(lightgray) & *(lightgray)  & *(white) & *(white)  \\
\none & \none & *(lightgray) & *(lightgray)  & *(white) & *(white)  \\
\none & \none & *(lightgray) & *(white) & *(white) \\
*(lightgray) 1 & *(lightgray) & *(lightgray) & *(white) 2  \\
\end{ytableau}
\qquad
\substack{\text{flip} \\ \longleftrightarrow  }
\qquad
\begin{ytableau}
\none & \none & \none & \none & \none & \none & *(white) \\
\none & \none & \none & *(white) & *(white) & *(white) & \\
\none & \none & *(white) & *(white)  & *(lightgray)  & *(lightgray)  \\
\none & \none & *(white) 1 & *(lightgray) & *(lightgray)  \\
*(lightgray) 2 & *(lightgray) & *(lightgray) & *(lightgray)  \\
\end{ytableau}
\end{align}

By analyzing the Littlewood map (see \cite{Pak2000Ribbon})
one can verify that there is at least one
pair of boxes, one from each strip, with the same content. Furthermore, we know that
both the strips have the same length $d$.
From these observations together
with the properties of the Littlewood map, one can show that
the ``outside strip'' is moved to the ``inside'' of the second strip, as shown in \eqref{eq:flip}.
See also \cref{fig:BSTExample}, where strips labeled $(6,7)$ and $(9,10)$ have been flipped.

We are now ready to state the main property of flips.
\begin{lemma}[{\cite[Lem. 4.1]{Pak2000Ribbon}}]
 Suppose $B$ and $B'$  in $\BST(\lambda/\mu,d)$ are related by a flip.
 Then $(-1)^{\height(B)}= (-1)^{\height(B')}$.
\end{lemma}
\begin{proof}
 Since the height is the sum of the heights of the individual strips,
 it suffices to consider the strips labeled $i$ and $i+1$.
 If these strips are disconnected, then the flip does not alter the
 height of the resulting border-strip tableau.
 It remains to prove the assertion when the strips are connected as in \eqref{eq:flip}.

The two strips form a skew Young tableau, which
can be divided into three parts:
a (non-empty) middle part which consists of all pairs of boxes $(b_1,b_2)$ whose contents are equal,
a part to the left of this middle part and a part to the right of this middle part.
A flip fixes the left and right part while it swaps the boxes
in the middle part. Hence, the strips either both increase or both decrease in height by $1$ and the lemma follows.
\end{proof}

Together with \cref{thm:rootEvaluationGivesCharacter} we can now
deduce the main result of this section.

\begin{corollary}\label{cor:characterAsBST}
Let $\lambda/\mu$ be a skew shape of size $n=dm$.
Then the signed sum
\begin{equation}\label{eq:mnRuleCancellationFree}
\chi^{\lambda/\mu}((m^d)) =
\sum_{B \in \BST\left(\lambda/\mu, m \right)} (-1)^{\height(B)}
\end{equation}
in the Murnaghan--Nakayama rule~\cref{thm:mnRule} is cancellation-free.
In particular, we have that
\[
  f^{\lambda/\mu}(\xi^d) = \chi^{\lambda/\mu}((m^d)) = \varepsilon |\BST(\lambda/\mu,m)|,
\]
where $\xi$ is a primitive $n^\thsup$ root of unity and
$\varepsilon=(-1)^{\height(B)}$ for any
$B\in \BST\left(\lambda/\mu, m \right)$.
\end{corollary}
Observe that the non-skew case of the above corollary was proved
already in \cite[Thm. 2.7.27]{JamesKerber1984} using the abacus model.

\section{Bounds on the number of border-strip tableaux}\label{sec:BSTBounds}

The goal of this section is to prove the following theorem.
\begin{theorem}\label{thm:BST-bound-skew}
  Let $\lambda/\mu$ be a skew shape with $n$ boxes and let $k$ be a
  positive integer with $k\mid n$.  Suppose that
  $|\BST(\lambda/\mu, k)| \geq 2$.  Then
  \[
  |\BST(\lambda/\mu, k)| \geq %
  \sum_{d\mid\frac{n}{k}, d > 1} |\BST(\lambda/\mu, dk)|.
  \]
  Additionally, the inequality holds if $n/k$ is a prime number.
\end{theorem}
\begin{example}
  For $\lambda=(10,\!1^2) \vdash 12$ we have $|\BST(\lambda, 3)| = 1$ and
  $|\BST(\lambda, 6)|+|\BST(\lambda, 12)| = 2$.  By contrast, for
  $\lambda=(9,\!1^3) \vdash 12$, we have $|\BST(\lambda, 1)|=165$,
  $|\BST(\lambda, 2)|=5$, $|\BST(\lambda, 3)|=3$ and
  $|\BST(\lambda, 4)|=|\BST(\lambda, 6)|=|\BST(\lambda, 12)|=1$, and
  therefore
  \[
  |\BST(\lambda, k)| \geq \sum_{d\mid \frac{n}{k}, d > 1} |\BST(\lambda, d k)|
  \]
  for all $k$.
\end{example}
\begin{remark}
  For $k=1$, apart from the single row and single column partitions,
  there are only three shapes $\lambda/\mu$ where equality is
  attained: $(2^2)$, $(3^2)$ and $(2^3)$.  Other than that, the
  minimal difference between the two sides of the inequality is
  attained for hooks of the form $(n-1,\!1)$.  In this case it equals
  $n-\tau(n)$, where $\tau(n)$ is the number of divisors of $n$.
\end{remark}

Our strategy is to reduce the theorem to the case of straight shapes
and $k=1$, which we prove in \cref{sec:straight-case}, employing
a bound due to S.~Fomin \& N.~Lulov.

In \cref{sec:skew-case-1} we extend this to the case of skew shapes
and $k=1$, essentially using the Littlewood--Richardson rule.

Finally, in \cref{sec:skew-case} we deduce the general case from
the inequality with $k=1$, using a bound on the quotient of a
multinomial coefficient and a multinomial coefficient with stretched
entries proved in \cref{sec:bounds-mult-coeff} and \cref{lem:stripHook} and \cref{lem:shapeAbacusPrimes}.

\subsection{Bounds on multinomial coefficients}
\label{sec:bounds-mult-coeff}
We first prove two inequalities related to multinomial coefficients. For this we use the approximation due to H.~Robbins.

\begin{proposition}[\cite{Robbins1955}]\label{thm:factorialBounds}
For any positive integer $n$,
  \begin{gather}\label{eq:robbins}
    n! = \sqrt{2\pi} n^{n+1/2} e^{-n+r_n}
    \quad
    \text{ for some }
    \quad
    \frac{1}{12n+1} < r_n < \frac{1}{12n}.
  \end{gather}
\end{proposition}

\begin{lemma}\label{lem:multinomQuotientIneq}
  For any positive integer $d$ and positive integers
  $m_1,\dotsc,m_k$ summing to $m$,
  \begin{equation}\label{eq:multinomQuotientIneq}
  \frac{
    \binom{dm}{d m_1,\dotsc,d m_k} }{ \binom{m}{m_1,\dotsc,m_k} } %
    \geq %
    \frac{1}{d^{(k-1)/2}}\left(\frac{m^m}{\prod_{j=1}^k m_j^{m_j}}\right)^{d-1}.
\end{equation}
\end{lemma}
\begin{proof}
  For $d=1$ or $k=1$, the statement is trivial, so we can assume $d>1$ and $k>1$.

The left hand side of \eqref{eq:multinomQuotientIneq} is
equal to
\[
\frac{(d m)!}{\prod_j (d m_j)!}\frac{\prod_j (m_j)!}{m!},
\]
which by \eqref{eq:robbins} is then larger than or equal to
\begin{align*}
 \frac{m^{(d-1)m}}{d^{(k-1)/2}\prod_i m_i^{(d-1)m_i}} %
   \exp\left(\frac{1}{12dm+1}-\frac{1}{12m}+\sum_j \frac{1}{12m_j+1}
   - \frac{1}{12dm_j}\right).
  \end{align*}
  It remains to show that, for $\varepsilon=1/12$,
  \[
    \frac{1}{dm+\varepsilon}+\sum_j \frac{1}{m_j+\varepsilon}
    \geq \frac{1}{m} + \sum_j\frac{1}{dm_j},
  \]
  provided $k>1$ and $d>1$.  Set $M = m\sum_j \frac{1}{m_j}$ and notice
  that $M\geq k\geq 2$.  Furthermore,
  $m\sum_j \frac{1}{m_j+\varepsilon} \geq
  \frac{m}{1+\varepsilon}\sum_j \frac{1}{m_j} =
  \frac{1}{1+\varepsilon} M$ and
  $\frac{dm}{dm+\varepsilon}\geq\frac{1}{1+\varepsilon}$.  Thus, it
  suffices to prove that, for $d\geq 2$, $M\geq 2$ and
  $\varepsilon=1/12$,
  \[
    \frac{1}{1+\varepsilon}(1 + dM) \geq d + M.
  \]
  This can be seen, for example, by replacing $d$ with $2+\tilde{d}$
  and $M$ with $2+\tilde{M}$.
\end{proof}

\begin{corollary}\label{cor:binomial-approximation}
  For any integer $d>1$ and $k>1$ positive integers $m_1,\dots,m_k$ summing to m,
  \[
    \frac{ \binom{dm}{d m_1,\dotsc,d m_k} }{ \binom{m}{m_1,\dotsc,m_k} } %
    \geq %
    \prod_{j=1}^k (\tau(d m_j)-1),
  \]
  where $\tau(n)$
  is the number of divisors of $n$.
\end{corollary}
\begin{proof}
We use the inequality $n/2 \geq \tau(n)-1$, valid for $n\geq 1$.
By \cref{lem:multinomQuotientIneq} it is then sufficient to show that
\begin{align*}
  \frac{1}{d^{(k-1)/2}} %
  \left(\frac{m^m}{\prod_{j=1}^k m_j^{m_j}}\right)^{d-1} %
  &> \left(\frac{d}{2}\right)^k\prod_{j=1}^k m_j, \\
\intertext{or, equivalently,}
  \left( \prod_{j=1}^k \frac{1}{m_j} \right)
  \left(\frac{m^m}{\prod_{j=1}^k m_j^{m_j}}\right)^{d-1} %
  &>
  \frac{1}{\sqrt{d}}
  \left(\frac{d}{2}\right)^k
  d^{k/2}. \\
\end{align*}
We will show the stronger inequality
  \begin{equation}\label{eq:binom-lemma-1}
    \left(\frac{m^m}{\prod_{j=1}^k m_j^{m_j+1}}\right)^{\frac{1}{k}}%
    > \left(\frac{d\sqrt{d}}{2}\right)^{\frac{1}{d-1}}.    \tag{$\ast$}
  \end{equation}
  It is not hard to see that the right hand side of this inequality,
  as a real function of $d>1$, attains its maximum between $3$ and $4$
  and is unimodal.  By direct inspection we see that for integral $d$
  the maximum of the right hand side is attained at $d=3$, where it
  is $\left(\frac{27}{4}\right)^{\frac14} \approx 1.61$.

  \def\m2{\check m}
  To find the minimum of the left hand side of \eqref{eq:binom-lemma-1},
  we consider the function
  \[
  h(z) = \frac{(z+\m2)^{z+\m2}}{z^{z+1}}, \text{ where $z\geq 1$.}
  \]
  For $\m2=1$ we have $h(z)=(1+1/z)^{z+1}$, which is
  strictly decreasing towards Euler's number $e$ as $z$ increases.  For $\m2\geq 2$
  we show that $h$ is strictly increasing.  Indeed, the derivative of
  $\ln h(z)$ equals
  \[
  \ln\left(1+\frac{\m2}{z}\right) - \frac{1}{z}.
  \]
  This expression is positive for $\m2\geq 2$ and $z\geq 1$, since we
  have
  \begin{align*}
    \exp\left(\frac{1}{z}\right) %
    &= 1 + \frac{1}{z}\left(1 + \sum_{k\geq 2}\frac{1}{k! z^{k-1}}\right)\\ %
    &\leq 1 + \frac{1}{z}\left(1 + \sum_{k\geq 2}\frac{1}{k!}\right) %
      = 1 + \frac{e - 1}{z} %
      < 1 + \frac{\m2}{z}.
  \end{align*}

  For $k=2$ and $m_2=1$, the left hand side
  of~\eqref{eq:binom-lemma-1} equals $\sqrt{h(m_1)}$ with $\m2=1$.
  It is thus strictly larger than $\sqrt{e}\approx 1.64$, which in turn
  is larger than $\left({27}/{4}\right)^{1/4}$, the maximum of the
  right hand side of~\eqref{eq:binom-lemma-1}.

  For $k=2$ and $m_2>1$ and for $k\geq 3$ the analysis of $h$ implies
  that the left hand side of~\eqref{eq:binom-lemma-1} is strictly
  increasing in each of the variables $m_i$, $1\leq i\leq k$, because
  it equals
  \[
    \left(\frac{h(m_i)}{\prod_{j\neq i} m_j^{m_j+1}}\right)^{1/k},
  \]
  with $\m2=\sum_{j\neq i} m_j$.  For $k=2$ and $m_2>1$, this
  expression is minimized at $m_1=1$, where it is larger than
  $\sqrt{e}$ as shown above.  For $k>2$ the minimum is attained at
  $m_1=\dotsb=m_k=1$, and is equal to $k$.
\end{proof}

\subsection{The bound for standard Young tableaux of straight shape}
\label{sec:straight-case}

The goal of this subsection is to prove the special case of \cref{thm:BST-bound-skew} where $\mu=\emptyset$ and $k=1$.
Note that we have the equivalence
\[
 |\BST(\lambda, 1)| \geq  \sum_{d\mid n, d > 1} |\BST(\lambda, d)|
 \quad\iff\quad
 \frac{ \sum_{d\mid n} |\BST(\lambda, d)| }{|\BST(\lambda, 1)| } \leq 2 \tag{$\ast$}\label{eq:ast}.
\]
For the remainder of this subsection we focus on proving the latter inequality
for $\lambda \notin \{ (n),(1^n)\}$.
We shall first make use the following theorem by S.~Fomin and N.~Lulov.
\begin{theorem}[\cite{FominLulov1997}]
For any partition $\lambda \vdash n$, we have
\begin{equation}\label{eq:qndDef}
  |\BST(\lambda, d)| \leq Q(n, d) \cdot  |\BST(\lambda, 1)|^{1/d}
  \quad
  \text{ where }
  \quad
  Q(n, d)\coloneqq \sqrt[d]{\frac{d^n}{\binom{n}{n/d,\dots,n/d}}}.
\end{equation}
\end{theorem}
We introduce the auxiliary function $\fbound_n(x)$ as
\begin{equation}\label{eq:fnx}
  \fbound_n(x) \coloneqq \sum_{d\mid n} Q(n, d) x^{\frac{1}{d}-1}.
\end{equation}
By plugging $x = |\BST(\lambda,1)|$ into \eqref{eq:fnx},
and using \eqref{eq:qndDef}, we have that
\begin{align*}
\fbound_n(|\BST(\lambda,1)|) &= \sum_{d\mid n} Q(n, d) |\BST(\lambda,1)|^{\frac{1}{d}-1} \\
      &= \frac{ \sum_{d\mid n} Q(n, d) |\BST(\lambda,1)|^{1/d} }{|\BST(\lambda,1)|} \\
 &\geq \frac{ \sum_{d\mid n} |\BST(\lambda,d)| }{|\BST(\lambda,1)|}.
\end{align*}
Hence, if we can show that $\fbound_n(x) \leq 2$ for suitable values
of $x$ and $n$ we obtain the second inequality in \eqref{eq:ast}.

\begin{lemma}\label{prop:mainStraightStatement}
The inequality
\[
\frac{ \sum_{d\mid n} |\BST(\lambda, d)| }{|\BST(\lambda, 1)| } \leq 2
\]
holds for all partitions $\lambda\vdash n$ other than $(n)$ and $(1^n)$ with $n$ composite.
\end{lemma}
\begin{proof}
We consider several separate cases.
The cases when $n$ is a prime number are trivial.

\textbf{Case $\lambda = (n-1,1)$ or $\lambda=(2, 1^{n-1})$.}
In this case,
  \[
  |\BST(\lambda, d)| =
  \begin{cases}
    n-1 &\text{if $d=1$}\\
    1 &\text{otherwise},
  \end{cases}
  \]
  and we observe that $n-1 \geq \tau(n)-1$, where $\tau(n)$ is the
  number of divisors of~$n$.

\textbf{Case $|\lambda|\leq 8$.}
The remaining $14$ partitions (and their conjugates) not covered previously
can be verified by hand.

\textbf{Case $|\lambda|\geq 9$.}
As we noted before, it suffices to show that $\fbound_n\left(  |\BST(\lambda, 1)|  \right) \leq 2$.
It then suffices to prove the following three properties, whenever $n \geq 9$:
\begin{itemize}
 \item the function $x \mapsto \fbound_n(x)$ is strictly decreasing for fixed $n$
 \item $\fbound_n\left(  \frac{n^2}{3} \right) \leq 2 $ and
 \item $|\BST(\lambda, 1)| \geq \frac{n^2}{3}$ for $\lambda \notin \{ (n),(1^n), (n-1,1), (2,1^{n-2})\}$.
\end{itemize}
The first item is obvious from the definition of $\fbound_n(x)$,
as all exponents of $x$ are negative.
The second item is proved later in \cref{lem:ngeq9ineq}.
We proceed by verifying the third item.
Suppose that $|\lambda|\geq 9$ and that $\lambda$ is not of the excluded shapes.
We prove the statement by induction over $|\lambda|$:
The computer verifies the base cases $|\lambda| \in \{9,10\}$ easily.
We now consider some \emph{exceptional shapes}.
If $\lambda$ is a hook of the form $(n-b,1^b)$ with $2 \leq b \leq n-3$,
then $|\BST(\lambda, 1)| = \binom{n-1}{b}$.
Moreover, if $\lambda = (n-2,2)$ or $\lambda = (2^2,1^{n-4})$
then $|\BST(\lambda, 1)| = \frac{n(n-3)}{2}$.
In both cases the inequality is true for $n\ge 9$.
Now, \emph{if $\lambda$ is a rectangle and $n\ge 11$}, we can either remove
two boxes from the last row, or the last column, to obtain $\mu$ and $\nu$,
respectively. A simple bijective argument shows that
\[
 |\BST(\lambda, 1)| =|\BST(\mu, 1)| + |\BST(\nu, 1)|,
\]
and by induction, $|\BST(\lambda, 1)| \geq 2(n-2)^2/3 \geq n^2/3$
where the last inequality is true for $n\geq 7$.
Note that $\mu$ and $\nu$ are not of the excluded shapes.

For \emph{the case when $\lambda$ is not a rectangle or an exceptional shape},
a similar argument works by removing two different corners of the
shape $\lambda$ (possible since the shape is not a rectangle)
in order to obtain $\mu$, $\nu$.
These smaller shapes are now either covered by the base cases or are hook shapes.
They are not of the excluded shapes (this could only happen for the exceptional shapes) and we have
\[
 |\BST(\lambda, 1)| \geq |\BST(\mu, 1)| + |\BST(\nu, 1)|.
\]
A similar argument as above concludes the proof.
\end{proof}

\begin{lemma}\label{lem:Qleqsqrtnineq}
For positive integers $d\mid n$, we have
\begin{equation}\label{eq:Qleqsqrtn}
Q(n, d) \leq \sqrt{n},
\end{equation}
where $Q(n,d) =  \left( \frac{d^n}{ n!/ ((n/d)!)^d } \right)^{1/d}$ is as in \eqref{eq:qndDef}.
\end{lemma}
\begin{proof}
First we prove~\eqref{eq:Qleqsqrtn} for $n \in \{1, 2\}$ by direct inspection.
For $n\ge 3$ we use again Robbins' approximation~\eqref{eq:robbins} from \cref{thm:factorialBounds}.
We have that
\begin{align*}
 Q(n,d) &= \left( \frac{d^n}{ n!/ ((n/d)!)^d } \right)^{\frac{1}{d}} =
 \frac{d^{n/d}(n/d)!}{(n!)^{1/d}}
\\
&=
(2\pi n)^{-\frac{1}{2d}} \left(2\pi\frac{n}{d}\right)^{\frac{1}{2}}\exp\left(r_{n/d}-\frac{r_n}{d}\right)
 .
\end{align*}
By plugging this into~\eqref{eq:Qleqsqrtn} and simplifying the expressions it suffices to show that
\begin{equation}\label{eq:QleqsqrtnReduced}
\exp(2ndr_{n/d}-2n r_n) \le \left(2\pi n M_d \right)^n,
\end{equation}
where $M_d \coloneqq \left(\frac{d}{2\pi}\right)^d$. By now using the approximations $r_{n/d} < \frac{d}{12n}$ and $r_n>0$ we have
\[
2ndr_{n/d}-2n r_n \le 2ndr_{n/d} \le \frac{2nd^2}{12n} = \frac{d^2}{6}
\]
and thus the left hand side of~\eqref{eq:QleqsqrtnReduced} can be bound with $\exp(\frac{d^2}{6})$.

\textbf{Claim:}
For positive integers, $n\ge 3$, $n\ge d$, we have
\begin{equation}\label{eq:nge3ineq}
\exp\left(\frac{d^2}{6}\right) \le \left(2\pi n M_d \right)^n.
\end{equation}

\textbf{Case $d\ge 8$}. In this case we have $\exp(1/6) < \frac{d}{2\pi}$,
and therefore $\exp(\frac{d^2}{6})< M_d^d \le M_d^n \le \left(2\pi n M_d\right)^{d}$.
This proves~\eqref{eq:nge3ineq}.

\textbf{Case $1\le d \le 7$.}
Direct inspection shows that $2\pi M_d \geq 2\pi M_2 = 2/\pi$,
so for $n\ge 2$ we have $2\pi n M_d > 1$ and therefore  $(2\pi n M_d)^n$ is strictly increasing in $n$.

Finally, direct inspection now verifies that~\eqref{eq:nge3ineq} is already satisfied for $n=3$ and each $d \in \{1,\dots, 7\}$.
This proves the claim and concludes the proof.
\end{proof}

\begin{lemma}\label{lem:ngeq9ineq}
 For integers $n \geq 9$ we have $\fbound_n\left(\frac{n^2}{3}\right)\leq 2$.
\end{lemma}
\begin{proof}
We first verify the inequality for all $n$ in the range $9 \leq n \leq 120$.
This can be done using a computer.
For $n \geq 121$, we shall bound $\fbound_n(x)$ from above by $g_n(x)$,
defined as
  \[
  g_n(x) = 1 + \sqrt{n} x^{-\frac{1}{2}} + 2 n x^{-\frac{2}{3}}.
  \]

A simple computation shows that for $n=121$, we have that $g_n\left(\frac{n^2}{3}\right)\approx 1.999$.
Hence, if we can show that
\begin{itemize}
 \item  $\fbound_n(x) \leq g_n(x)$ for all $x \geq 1$ and
 \item $n \mapsto g_n\left(\frac{n^2}{3}\right)$ is strictly decreasing,
\end{itemize}
then we are done.
We proceed with the first point.
Recall that the number of divisors of $n$ is at most $2\sqrt{n}$
for $n\geq 1$.  We thus obtain
  \begin{align*}
    \fbound_n(x) &= \sum_{d\mid n} Q(n, d) x^{\frac{1}{d}-1}\\
           &= 1 + \sum_{d\mid n, d>1} Q(n, d) x^{\frac{1}{d}-1}\\
         \text{\{by \cref{lem:Qleqsqrtnineq}\}}  &\leq 1 + \sum_{d\mid n, d>1} \sqrt{n} x^{\frac{1}{d}-1}\\
           &\leq 1 + \sqrt{n} x^{-\frac{1}{2}} %
             + \sqrt{n} \sum_{d\mid n, d>1} x^{\frac{1}{3}-1}\\
           &\leq 1 + \sqrt{n} x^{-\frac{1}{2}} %
             + 2 n x^{-\frac{2}{3}}\\
              &= g_n(x).
  \end{align*}
For the second point, it is enough to note that
\[
 g_n(n^2/3) = 1+\frac{\sqrt{3}}{\sqrt{n}} + \frac{2 \times 9^{1/3}}{\sqrt[3]{n}}
\]
which is obviously decreasing in $n$.
\end{proof}

\subsection{The bound for skew standard Young tableaux}
\label{sec:skew-case-1}

We now extend the result of the previous section to skew shapes.
\begin{lemma}\label{thm:BST-bound-1}
  Let $\lambda/\mu$ be a skew shape with $n$ boxes.  Then
  \[
  |\BST(\lambda/\mu, 1)| \geq %
  \sum_{d\mid n, d > 1} |\BST(\lambda/\mu, d)|.
  \]
  if and only if $\lambda/\mu$ is neither the partition $(n)$ nor the
  partition $(1^n)$.
\end{lemma}
\begin{proof}
We distinguish two cases:

\textbf{Case 1:} $\lambda/\mu$ has at least two boxes in some row and in some column.

We prove the following sequence of inequalities:
\begin{align}
|\BST(\lambda/\mu, 1)| %
&= \sum_{\nu\vdash n} c^\lambda_{\mu,\nu} |\BST(\nu, 1)|\nonumber \\
&\geq
 \sum_{\nu\vdash n} %
c^\lambda_{\mu,\nu} \sum_{d\mid n, d > 1} |\BST(\nu, d)| \label[ineq]{eq:inequality1}\\
&\geq
\sum_{d\mid n, d > 1} |\BST(\lambda/\mu, d)|. \label[ineq]{eq:inequality2}
\end{align}

We first prove \cref{eq:inequality1}, for each summand separately.
That is, we show that for all partitions $\nu \vdash n$,
\[
c^\lambda_{\mu,\nu} |\BST(\nu, 1)| \geq
c^\lambda_{\mu,\nu} \sum_{d\mid n, d > 1} |\BST(\nu, d)|.
\]
Since $\lambda/\mu$ has at least two boxes in the same row and two boxes in the same column,
we can apply \cref{lemma:vanishing LR-coeffs}. It follows that $c^\lambda_{\mu,(1^n)}=c^\lambda_{\mu,(n)}=0$.
Thus, the inequality holds for $\nu = (1^n)$ and $(n)$.
For all other partitions $\nu \vdash n$, the inequality follows from \cref{prop:mainStraightStatement}.

We now change the order of summation and prove \cref{eq:inequality2}, again separately for each summand.
That is, for fixed $d>1$ with $dm = n$, we show
\[
\sum_{\nu\vdash n} c^\lambda_{\mu,\nu} |\BST(\nu, d)| \geq |\BST(\lambda/\mu, d)|.
\]
Indeed, we have
\begin{align*}
  \sum_{\nu\vdash n} c^\lambda_{\mu,\nu} \left|\BST(\nu, d) \right| %
  &= \sum_{\nu\vdash n} c^\lambda_{\mu,\nu} \left| \chi^{\nu}((d^m)) \right| \\%
  &\geq \left|\chi^{\lambda/\mu}((d^m)) \right| \\
  &= \left|\BST(\lambda/\mu, d) \right|.
\end{align*}
The two equalities follow from \cref{cor:characterAsBST}, whereas the
inequality is obtained by taking absolute values on both sides of the
Littlewood--Richardson rule for characters~\eqref{eq:characterLRRule},
evaluated at $(d^m)$ and applying the triangle inequality.

\textbf{Case 2:} all columns or all rows of $\lambda/\mu$ contain at most one box.

By symmetry we may assume
that every connected component of $\lambda/\mu$ is a single row.
Let the lengths of these rows be $n_1,n_2,\dotsc,n_r$.
We have $\BST(\lambda/\mu,d)=\emptyset$ unless all $n_i$ are multiples of $d$.
In this case we find by explicit enumeration that
\[
  |\BST(\lambda/\mu, d)| = \binom{n/d}{n_1/d,n_2/d,\dotsc,n_r/d}.
\]
It then suffices to prove that
\[
 \binom{n}{n_1,n_2,\dotsc,n_r} \geq \sum_{d>1, d\mid \gcd(n_1, n_2, \dotsc,n_r)} \binom{n/d}{n_1/d,n_2/d,\dotsc,n_r/d}.
\]
This inequality is an easy consequence of \cref{cor:binomial-approximation}.
\end{proof}

\subsection{The general case}
\label{sec:skew-case}

We are now ready to prove \cref{thm:BST-bound-skew} itself.

\begin{proof}[Proof of \cref{thm:BST-bound-skew}]
  Let $(\nu^1/\kappa^1,\dots,\nu^k/\kappa^k)$ be the skew
  $k$-quotient of $\lambda/\mu$.  We first establish the inequality
  \[
    (\tau(|\nu^i/\kappa^i|)-1)|\BST(\nu^i/\kappa^i, 1)| %
    \geq \; %
    \sum_{\substack{ d \; \mid \; |\nu^i/\kappa^i| \\ d > 1 }} %
    \; %
    |\BST(\nu^i/\kappa^i, d)|.
  \]
  If $\nu^i/\kappa^i$ is neither the single row nor the single column
  partition, the bound for skew standard Young tableaux,
  \cref{thm:BST-bound-1}, applies.  Moreover, in this case
  $|\nu^i/\kappa^i| \geq 3$ and therefore
  $\tau(|\nu^i/\kappa^i|)-1\geq 1$.  Otherwise,
  $|\BST(\nu^i/\kappa^i, d)| = 1$ for all $d\mid |\nu^i/\kappa^i|$, and
  the inequality holds trivially.

  Thus, setting $g = \gcd(|\nu^1/\kappa^1|,\dots,|\nu^k/\kappa^k|)$,
  \begin{align*}
    \prod_{\substack{i=1\\ \nu^i/\kappa^i\neq\emptyset}}^k %
    \left(\tau(|\nu^i/\kappa^i|)-1\right) %
    |\BST(\nu^i/\kappa^i, 1)|
 &\geq
   \prod_{\substack{i=1\\ \nu^i/\kappa^i \neq \emptyset}}^k %
     \sum_{\substack{ d \; \mid \; |\nu^i/\kappa^i| \\ d > 1 }} %
    |\BST(\nu^i/\kappa^i, d)| \\%
 &\geq
   \sum_{\substack{d\mid g \\ d>1}} %
    \prod_{\substack{i=1\\ \nu^i/\kappa^i \neq \emptyset}}^k |\BST(\nu^i/\kappa^i, d)| \\ %
 &=
   \sum_{\substack{d\mid g \\ d>1}}  %
    \prod_{i=1}^k |\BST(\nu^i/\kappa^i, d)| \\ %
\{\text{By \cref{lem:stripHook} }\} &=
   \sum_{\substack{d\mid g \\ d>1}}  %
   \frac{
    |\BST(\lambda/\mu, dk)| }{
		\binom{\sum_i|\nu^i/\kappa^i|/d}%
		{|\nu^1/\kappa^1|/d,\dotsc,|\nu^k/\kappa^k|/d}
    }.
  \end{align*}
  Note that, for any $d\geq 1$, there is exactly one border-strip
  tableaux having empty shape: $|\BST(\emptyset, k)| = 1$.  Suppose
  that $g = \gcd(|\nu^1/\kappa^1|,\dots,|\nu^k/\kappa^k|) > 1$ and
  there are at least two non-empty skew shapes among
  $\nu^1/\kappa^1,\dots,\nu^k/\kappa^k$.  Then we can apply the above
  inequality and \cref{cor:binomial-approximation} and
  obtain
  \begin{align*}
    |\BST(\lambda/\mu, k)| %
    &=
      \binom{\sum_i|\nu^i/\kappa^i|}{|\nu^1/\kappa^1|,\dotsc,|\nu^k/\kappa^k|} %
      \prod_{i=1}^k |\BST(\nu^i/\kappa^i, 1)| \\
    &\geq
      \binom{\sum_i|\nu^i/\kappa^i|}%
      {|\nu^1/\kappa^1|,\dotsc,|\nu^k/\kappa^k|}
      \prod_{\substack{i=1\\ \nu^i/\kappa^i \neq\emptyset}}^k \left(\tau(|\nu^i/\kappa^i|)-1\right)^{-1} \\%
    &\qquad \sum_{{d\mid g, d>1}} %
    |\BST(\lambda/\mu, dk)|
    \binom{\sum_i|\nu^i/\kappa^i|/d}%
    {|\nu^1/\kappa^1|/d,\dotsc,|\nu^k/\kappa^k|/d}^{-1} \\ %
    &\geq
      \sum_{{d\mid g, d>1}} |\BST(\lambda/\mu, dk)| \\%
    &=
      \sum_{d>1} |\BST(\lambda/\mu, dk)|.%
  \end{align*}
  If $g = \gcd(|\nu^1/\kappa^1|,\dots,|\nu^k/\kappa^k|) = 1$, the inequality is
  trivially true.

  If there is precisely one non-empty skew shape $\nu/\kappa$
  among $\nu^1/\kappa^1,\dots,\nu^k/\kappa^k$, we have
  $|\BST(\lambda/\mu, d k)| = |\BST(\nu/\kappa, d)|$ for all $d$ by
  \cref{lem:stripHook}.

  If $\nu/\kappa$ is neither $(n/k)$ nor $(1^{n/k})$, 
  \cref{thm:BST-bound-1} applies and we have
  \[
    |\BST(\lambda/\mu, k)| = |\BST(\nu/\kappa, 1)| %
    \geq \sum_{d\mid \frac{n}{k}, d>1} |\BST(\nu/\kappa, d)| %
    = \sum_{d\mid \frac{n}{k}, d>1} |\BST(\lambda/\mu, dk)|.
  \]
  Otherwise, if $\nu/\kappa$ is either $(n/k)$ or $(1^{n/k})$,
  \cref{lem:shapeAbacusPrimes} implies that there is only one element
  in $\BST(\lambda/\mu, k)$.
\end{proof}

\section{Cyclic sieving for skew standard tableaux}\label{sec:mainCSP}

In this section we apply the bounds established in the previous
section and \cref{thm:alexanderssonAmini} to prove the existence
of several cyclic sieving phenomena for various families of skew
standard Young tableaux.

Let us first put the bound from \cref{thm:BST-bound-skew} into the
form required to apply \cref{thm:alexanderssonAmini}.
\begin{proposition}\label{cor:mainCSPInequality}
  Let $\lambda/\mu$ be a skew shape with $n$ boxes and let $m\in\setN$.
  Then, for all $k\in\setN$ with $k \mid n$,
  \[
    \sum_{d\mid\frac{n}{k}}\mu(d) |\BST(\lambda/\mu, d k)|^m \geq 0,
  \]
  or, equivalently,
  \[
  \sum_{d\mid k}\mu(k/d) |f^{\lambda/\mu}(\xi^d)|^m \geq 0.
  \]
\end{proposition}
\begin{proof}
  The equivalence of the two inequalities follows from
  \cref{cor:characterAsBST} and replacing $d$ with $\frac{n}{dk}$ and
  $k$ with $\frac{n}{k}$.  We prove the former inequality.

  If $|\BST(\lambda/\mu, k)|=1$, we also have
  $|\BST(\lambda/\mu, dk)|=1$ for any $d \mid \frac{n}{k} $ by
  \cref{lem:shapeAbacusPrimes}. Therefore,
  \[
    \sum_{d\mid \frac{n}{k}}\mu(d) |\BST(\lambda/\mu, d k)|^m =
    \sum_{d\mid \frac{n}{k}}\mu(d) =
    \begin{cases*}
    	1&if $n=k$\\
    	0&if $n\neq k$
    \end{cases*} \ge 0.
  \]
  This reasoning also covers the case $m=0$.

  Otherwise, since $\mu(1)=1$ and $\mu(d)\geq -1$, we have
  \begin{align*}
    \sum_{d\mid \frac{n}{k}} \mu(d) |\BST(\lambda/\mu, d k)|^m
    \geq |\BST(\lambda/\mu, k)|^m - \sum_{d\mid\frac{n}{k}, d > 1} |\BST(\lambda/\mu, d k)|^m \geq 0,
  \end{align*}
  where the final inequality follows from \cref{thm:BST-bound-skew}.
\end{proof}

\begin{remark}
  One might think that $|f^{\lambda/\mu}(\xi^d)|$ could be the
  number of fixed points of a group action, despite the fact that
  $|f^{\lambda/\mu}(q))|$ is not a polynomial.  However, this is not
  the case.

  For example, consider $\lambda=(2,1)$.  Then $f^\lambda(q) = q+q^2$
  and, for a $3^\rdsup$ root of unity $\xi$ we have
  $|f^\lambda(\xi^3)|=|\BST(\lambda,1)|=2$ and
  $|f^\lambda(\xi)|=|\BST(\lambda,3)|=1$, which is incompatible with the
  possible orbit sizes of a group action on a set with two elements.
  Indeed, for $k=3$ we obtain
  \[
    \frac{1}{k}\sum_{d\mid k} \mu(k/d) |f^\lambda(\xi^d)| = \frac{1}{3}(-1 + 2),
  \]
  which, by \cref{rmk:orbit-sizes}, would have to be an integer.
\end{remark}

Taking into account the previous remark, it makes sense to look for
shapes $\lambda/\mu$ such that the character $f^{\lambda/\mu}$
evaluated at roots of unity is nonnegative.
\begin{proposition}\label{prop:positive-cyclic-sieving}
  Let $\lambda/\mu$ be a skew shape with $n$ boxes and let
  $m\in\setN$.  Then there is a group action $\rho$ such that
  \[
  \left(
    \underbrace{\SYT(\lambda/\mu)\times\dots\times\SYT(\lambda/\mu)}_m,
    \langle \rho \rangle, f^{\lambda/\mu}(q)^m \right)
  \]
  exhibits the cyclic sieving phenomenon if and only if $m$ is even,
  or $m$ is odd and for each $k\in\setN$ with $k\mid n$ there is a
  tiling of $\lambda/\mu$ of even height with strips of size $k$.
\end{proposition}
\begin{remark}
  The case $m=2$ of this proposition does not extend to squares of arbitrary
  representations of the symmetric group.  For example, consider the
  representation with character $\chi^{(4)} + \chi^{(2,1^2)}$.  Its
  fake degree polynomial is $f(q) = 1+q^3+q^4+q^5$.  Then we obtain,
  for a primitive fourth root of unity $\xi$, that $f(\xi)^2 = 4$ and
  $f(\xi^2)^2 = 0$.  This violates the condition in
  \cref{thm:alexanderssonAmini} for $k=2$, because
  $\mu(2)f(\xi)^2 + \mu(1)f(\xi^2)^2 = -4$.
\end{remark}
\begin{proof}
  Let $\xi$ be an $n^\thsup$ primitive root of unity.  Then
  \cref{cor:mainCSPInequality} together with
  \cref{thm:alexanderssonAmini} ensures the existence of $\rho$,
  provided $f^{\lambda/\mu}(\xi^d)^m$ is nonnegative for all
  $d\mid n$.  Conversely, nonnegativity is a necessary condition
  because, given a group action $\rho$, the number of fixed points of
  $\rho^d$ equals $f^{\lambda/\mu}(\xi^d)^m$.

  It remains to consider the case of odd $m$.  By
  \cref{cor:characterAsBST} $f^{\lambda/\mu}(\xi^d)$, with
  $d=\frac{n}{k}$, is nonnegative if and only if there is a tiling of
  $\lambda/\mu$ of even height with strips of size~$k$.
\end{proof}

\begin{corollary}
  Let $\lambda = (a, 1^{n-a})$ be a hook-shaped partition of $n$.
  Then there is a group action $\rho$ such that
  $\left( \SYT(\lambda), \langle \rho \rangle, f^{\lambda}(q)
  \right)$ exhibits the cyclic sieving phenomenon if and only if $n$ and $a$ are odd and
  $a-1 \mod m$ is even for $m\mid n$, $1\leq m < a$.
\end{corollary}
\begin{proof}
  Suppose that $n$, and $a$ are odd, and $m\mid n$.  In particular,
  $m$ is odd, too.  Note that there is a unique tiling of a hook with
  border-strips of size $m$.  We have to show that the height of this
  tiling is even if and only if $a-1\mod m$ is even.

  Recall that the height of a tile is one less than the number of
  rows it spans.  If, and only if $a-1\mod m$ is even, the height of the tile
  covering the top left corner of the shape must be even: this tile
  must cover an odd number of boxes in the first row and, since its
  size $m$ is odd, an even number of boxes in the first column.
  Since the height of all other tiles is evidently even, too, so is
  the total height.

  If the parity of $n$ and $a$ is different, then the tiling with a
  single strip of size $n$ has height $n-a$, which is odd.  If both
  $n$ and $a$ are even, the tiling with two strips of size $n/2$ has
  odd height: if $a\leq n/2$, the height is $n-a-1$, otherwise the
  height is $n-1$.
\end{proof}
\begin{remark}
  According to the previous theorem, for $\lambda=(3,1^{n-3})$ a
  group action of order $n$ with character $f^\lambda(q) = q^{(n-2)(n-3)/2}\frac{[n-1]_q[n-2]_q}{[2]_q}$ exists for all
  odd $n>3$.  In this case, there should be one singleton orbit and
  $(n-3)/2$ orbits of size $n$.  Indeed, an appropriate group action
  can be constructed as follows:

  Identify a tableau with the two entries $x < y$ different from $1$
  in the first row.
  Note that $y-x \in \{1,2,\dotsc,n-2\}$, and only the pair $(2,n)$
  has difference $n-2$.
  We let the generator of the group action $\eta$ act as follows:
  \[
   \eta(x,y) \coloneqq
   \begin{cases}
    (2,n) &\text{ if } x=2,\; y=n, \\
    (x+2,y+2) &\text{ if } 2\leq x < y \leq n-2,\\
    (2,x+1) &\text{ if } y=n-1, \\
    (3,x+1) &\text{ if }x>2,\; y=n. \\
   \end{cases}
  \]
  We then note that if $(u,v)=\eta(x,y)$, then
  $v-u \in \{y-x, (n-2) - (y-x) \}$.
  This explains why there are $(n-3)/2$ orbits of length $n$.
  We leave the remaining details to the reader.
\end{remark}

\begin{remark}
  It turns out that one can determine the number of border strips
  $\lambda/\mu$ of size $n$ which carry a group action of order $n$
  and character $f^{\lambda/\mu}(q)$.  This will appear in a separate
  note~\cite{Pfannerer2020}.
\end{remark}

\medskip

A different way to ensure positivity of the character
$f^{\lambda/\mu}$ is to decrease the order of the cyclic group as in
\cref{rem:smallerGroupCSP}.
\begin{theorem}\label{thm:stretchedCSPtheorem}
Let $\lambda/\mu$ be a skew shape such that every row contains a multiple of $m$ boxes.
Then there is a cyclic group action $\rho$ of order $m$ such that
\begin{equation}\label{eq:stretchedCSP}
 \left( \SYT(\lambda/\mu), \langle\rho\rangle, f^{\lambda/\mu}(q) \right)
\end{equation}
exhibits the cyclic sieving phenomenon.
\end{theorem}
\begin{proof}
  By \cref{thm:alexanderssonAmini} it suffices to show that for a
  primitive $m^\thsup$ root of unity $\zeta$ and every $k\mid m$
  \[
    \sum_{d|k} \mu(k/d) f^{\lambda/\mu}(\zeta^d)\geq 0.
  \]
   Let $|\lambda/\mu| = d m$.
   By \cref{cor:mainCSPInequality}, we have
   \[
     \sum_{d|k} \mu(k/d) |f^{\lambda/\mu}(\xi^d)|\geq 0
   \]
   for an $n^\thsup$ root of unity $\xi$ and every $k\mid d m$.  Let
   $\zeta=\xi^\frac{n}{m}$.  Then, by \cref{cor:characterAsBST},
   \[
     f^{\lambda/\mu}(\zeta^d) %
     = f^{\lambda/\mu}(\xi^{d\frac{n}{m}}) %
     = (-1)^{\height(B)} |\BST\left(\lambda/\mu, \frac{m}{d}\right)|.
   \]
  Since the length of every row of $\lambda/\mu$ is a multiple of $m$, there is a
  filling with border-strips of size $\frac{m}{d} \mid m$, where every strip has height $0$.
\end{proof}
We remark that stretching shapes seems to be a fruitful
way to construct cyclic sieving phenomena, as was previously shown with fillings related to Macdonald polynomials by P.~Alexandersson \& J.~Uhlin~\cite{AlexanderssonUhlin2019}.
Possibly an adaptation of the approach presented here can be used to settle the following conjecture:
\begin{conjecture}[{\cite[Conj. 3.4]{AlexanderssonAmini2018}}]
 There is an action $\beta$ on the set of semi-standard Young tableaux $\SSYT(m\lambda/m\mu,k)$
 of order $m$ such that
 \[
  \left( \SSYT(m\lambda/m\mu,k), \langle \beta \rangle,
  \schurS_{m\lambda/m\mu}(1,q,q^2,\dotsc,q^{k-1})
  \right)
 \]
 exhibits the cyclic sieving phenomenon.
\end{conjecture}

For some shapes $\lambda/\mu$, the tiling may have odd height, but one
can multiply $f^{\lambda/\mu}(q)$ with $q^{n/2}$, provided that the size $n$ of $\lambda/\mu$ is even, to
obtain positivity at roots of unity.  An important example is the
case of rectangular shapes.  In this case, B.~Rhoades proved that
promotion, together with a natural $q$-analogue of the hook length formula exhibits the cyclic sieving phenomenon.  The following result is much weaker,
because it only establishes the existence of a group action, but it
is also much easier to prove, and illustrates the method.
\begin{theorem}[\cite{Rhoades2010}]
 Let $\lambda = a^b$ be a rectangular diagram with $n=ab$ boxes,
and set $\kappa(\lambda) \coloneqq \sum_{j} \binom{\lambda'_j}{2}$.
 Then there is a group action $\partial$ of order $n$
 such that
 \[
  \left( \SYT(\lambda), \langle \partial \rangle, q^{-\kappa(\lambda)} f^{\lambda}(q) \right)
 \]
exhibits the cyclic sieving phenomenon.
\end{theorem}
\begin{proof}
  It is a well-known result by R.~Stanley~\cite[Cor.~7.21.5]{StanleyEC2}, that
  \[
    q^{-\kappa(\lambda)} f^{\lambda}(q) = \frac{[n]_q!}{\prod_{\square \in \lambda} [h(\square)]_q}
  \]
  where $h(\square)$ is the hook-value of $\square$.  In particular,
  $q^{-\kappa(\lambda)} f^{\lambda}(q)$ is a polynomial.
  We must check that this is nonnegative whenever $q$ is an $n^\thsup$ root of unity.
  Suppose that $m\mid n$, $n=dm$ and let $\xi$ be a primitive $n^\thsup$ root of unity.
  \cref{cor:characterAsBST} implies that $f^{\lambda}(\xi^{d})$
  is non-zero only if and only if $\BST(a^b,m)$ is non-empty.  Using the abacus, one can show that $m \mid a$ or $m \mid b$ if and only if the $m$-core is empty, which, for straight shapes,  is equivalent to $|\BST(a^b,m)|>0$.
  From here, it is a straightforward exercise to show that $\kappa(\lambda) = ba(a-1)/2$
  and that $\xi^{-d\cdot ba(a-1)/2} f^{\lambda}(\xi^{d})$ is nonnegative for all $d \mid n$.

  Finally, \cref{cor:mainCSPInequality} and
  \cref{thm:alexanderssonAmini} gives the result.
\end{proof}

\section{Permutations and invariants of the adjoint representation of
  \texorpdfstring{$\mathrm{GL}_n$}{GLn}}\label{sec:gln}

In this section we apply our results to study the space of invariants
of tensor powers of the adjoint representation $\gl_r$ of the general
linear group $\GL_r$.

\begin{definition}
  The \defin{rotation} $\rot\sigma$ of a permutation
  $\sigma\in\symS_n$ is the permutation obtained by conjugating with
  the long cycle $\longcycle$.
\end{definition}
\begin{remark}
  Equivalently, if $M_\sigma$ is the permutation matrix corresponding
  to $\sigma$, then $M_{\rot\sigma}$ is obtained by removing the
  first column of $M_\sigma$ and appending it on the right, and then
  removing the first column and appending it at the bottom.

  Yet equivalently, let $D_\sigma$ be the chord diagram associated
  with $\sigma$, that is, the directed graph with vertices
  $\{1,\dotsc,n\}$ arranged counterclockwise on a circle, and arcs
  $(i,\sigma(i))$.  Then $D_{\rot\sigma}$ is the chord diagram
  obtained by rotating the graph clockwise.
  See \cref{fig:rotation-permutation} for an illustration.
\end{remark}
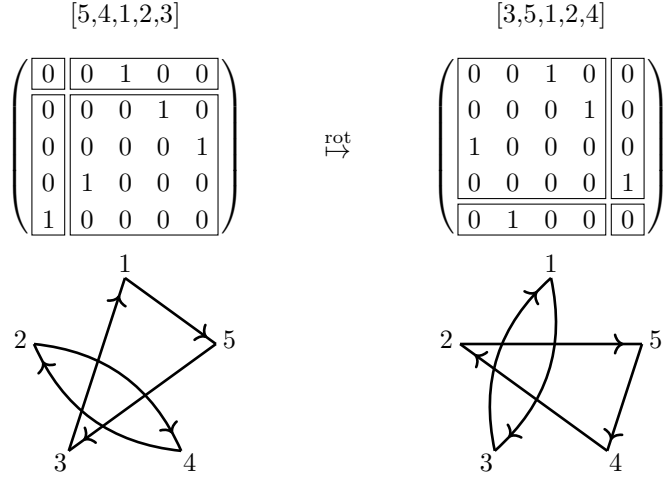
\begin{figure}[!ht]
  \tikzset{->-/.style={decoration={
        markings,
        mark=at position #1 with {\arrow{>}}},postaction={decorate}}}
\[
\begin{array}{ccc}
[5,\!4,\!1,\!2,\!3] \qquad && \qquad   [3,\!5,\!1,\!2,\!4]\\[1em]
\begin{tikzpicture}[baseline={([yshift=-1ex]current bounding box.center)}]
\matrix(m) [matrix of math nodes,
inner sep=0pt, column sep=0.25em,
row sep=0.25em,
nodes={inner sep=0.25em,text width=0.7em,align=center},
left delimiter=(,right delimiter=),
]
{%
0 & 0 & 1 & 0 & 0 \\
0 & 0 & 0 & 1 & 0 \\
0 & 0 & 0 & 0 & 1 \\
0 & 1 & 0 & 0 & 0 \\
1 & 0 & 0 & 0 & 0 \\
};%
\draw (m-1-1.north west) rectangle (m-1-1.south east);
\draw (m-2-2.north west) rectangle (m-5-5.south east);
\draw (m-1-2.north west) rectangle (m-1-5.south east);
\draw (m-2-1.north west) rectangle (m-5-1.south east);
\end{tikzpicture}\qquad &\stackrel{\rot}{\mapsto}& \qquad
\begin{tikzpicture}[baseline={([yshift=-1ex]current bounding box.center)}]
\matrix(m) [matrix of math nodes,
inner sep=0pt, column sep=0.25em,
row sep=0.25em,
nodes={inner sep=0.25em,text width=0.7em,align=center},
left delimiter=(,right delimiter=),
]
{%
0 & 0 & 1 & 0 & 0 \\
0 & 0 & 0 & 1 & 0 \\
1 & 0 & 0 & 0 & 0 \\
0 & 0 & 0 & 0 & 1 \\
0 & 1 & 0 & 0 & 0 \\
};%
\draw (m-5-5.north west) rectangle (m-5-5.south east);
\draw (m-1-1.north west) rectangle (m-4-4.south east);
\draw (m-5-1.north west) rectangle (m-5-4.south east);
\draw (m-1-5.north west) rectangle (m-4-5.south east);
\end{tikzpicture} \\[3em]
  \begin{tikzpicture}[line width=1pt, baseline={([yshift=-1ex]current bounding box.center)}]
    \node (a) [draw=none, minimum size=2.5cm, regular polygon, regular polygon sides=5] at (0,0) {};
    \foreach \n in {1,2,...,5}
        \path (a.center) -- (a.corner \n) node[pos=1.15] {$\n$};
    \draw[->-=.9] (a.corner 1) to (a.corner 5);
    \draw[->-=.9, bend left] (a.corner 2) to (a.corner 4);
    \draw[->-=.9] (a.corner 3) to (a.corner 1);
    \draw[->-=.9, bend left] (a.corner 4) to (a.corner 2);
    \draw[->-=.9] (a.corner 5) to (a.corner 3);
  \end{tikzpicture}
  \qquad && \qquad

    \begin{tikzpicture}[line width=1pt, baseline={([yshift=-1ex]current bounding box.center)}]
    \node (a) [draw=none, minimum size=2.5cm, regular polygon, regular polygon sides=5] at (0,0) {};
    \foreach \n in {1,2,...,5}
        \path (a.center) -- (a.corner \n) node[pos=1.15] {$\n$};
    \draw[->-=.9] (a.corner 5) to (a.corner 4);
    \draw[->-=.9, bend left] (a.corner 1) to (a.corner 3);
    \draw[->-=.9] (a.corner 2) to (a.corner 5);
    \draw[->-=.9, bend left] (a.corner 3) to (a.corner 1);
    \draw[->-=.9] (a.corner 4) to (a.corner 2);
  \end{tikzpicture}
\end{array}
\]
\caption{Rotation of $\pi = [5,\!4,\!1,\!2,\!3]$ as conjugation by the long cycle $(1,2,3,4,5)$, cyclic shift of the permutation matrix and rotation of the chord diagram. Note that $\shape([5,\!4,\!1,\!2,\!3])= (3,1^2)$ and $\shape([3,\!5,\!1,\!2,\!4])= (3,2)$}
\label{fig:rotation-permutation}
\end{figure}

The following theorem makes the character of rotation explicit.
\begin{theorem}[\cite{BarceloReinerStanton2008,Rhoades2010b,RubeyWestbury2014}]\label{thm:rhoadesPMatrix}
  \begin{equation}\label{eq:permMatrixCSP}
    \left( \symS_n, \langle \rot \rangle,  \sum_{\lambda \vdash n}   f^{\lambda}(q)^2  \right)
  \end{equation}
  exhibits the cyclic sieving phenomenon.
\end{theorem}
\begin{proof}
  Consider the adjoint representation of $\symS_n$, that is,
  $\symS_n$ acting on itself by conjugation, or relabelling.  It is
  well-known (see, e.g., \cite[Ex.~7.71a]{StanleyEC2}) that the
  character of this representation equals
  $\sum_{\lambda \vdash n} \chi^\lambda \bar\chi^\lambda$.  Since the
  restriction of the adjoint representation to the cyclic group
  generated by the long cycle $\longcycle$ is precisely the action
  $\rot$, the result follows from
  \cref{thm:rootEvaluationGivesCharacter}.
\end{proof}

\begin{definition}
Recall that the \defin{Robinson--Schensted correspondence} provides a bijection
\[
 \symS_n
 \leftrightarrow
 \{
 (P,Q) \in \SYT(\lambda)\times \SYT(\lambda) : \lambda \vdash n
 \}.
\]
The \defin{shape} $\shape(\sigma)$ of a permutation $\sigma$ is the
common shape of the standard Young tableaux $P$ and $Q$
corresponding to $\sigma$ under the Robinson--Schensted
correspondence.  We let \defin{$R_\lambda$} denote the set of permutations of shape $\lambda$.
\end{definition}

We are now ready to prove the first major result of this section.
\begin{corollary}\label{cor:rotationInvatiantStat}
  Let $P_n$ be the set of partitions of $n$.  Then there exists a map
  $\st:\symS_n\to P_n$ which is invariant under rotation and
  equidistributed with the Robinson--Schensted shape.  That is,
  \[
    \st\circ \rot = \st\qquad\text{and}\qquad %
    \sum_{\sigma\in\symS_n}  \schurS_{\st(\sigma)}(\xvec)= %
    \sum_{\sigma\in\symS_n}  \schurS_{\shape(\sigma)}(\xvec).
  \]
  Moreover, with
  $\symS_n^\lambda\coloneqq\{\pi\in\symS_n\mid \st(\sigma)=\lambda\}$,
  the triple
  \[
    (\symS_n^\lambda, \langle\rot\rangle, f^\lambda(q)^2)
  \]
  exhibits the cyclic sieving phenomenon.

\end{corollary}
\begin{remark}
  We stress that we are unable to present such a statistic explicitly.

  Note that the distribution of the Robinson--Schensted shape over all
  permutations is essentially the Plancherel measure.  Remarkably,
  \[
    \sum_{\sigma\in\symS_n}
    \schurS_{\shape(\sigma)}(\xvec)/f^{\shape(\sigma)}%
    = \powerSum_{1^n}(\xvec).
  \]
\end{remark}
\begin{proof}
By \cref{prop:positive-cyclic-sieving} there exists an action of the cyclic group of
order $n$ on $R_\lambda$ with character $(f^\lambda(q))^2$.
Let $\rho$ be the direct sum over all $\lambda\in P_n$ of these group actions.
Thus, $\left( \symS_n, \langle \rho \rangle,  \sum_{\lambda \vdash n}   f^{\lambda}(q)^2  \right)$
exhibits the cyclic sieving phenomenon.
Since $\rho$ acts on each $R_\lambda$ separately, we have
\begin{equation}\label{eq:shapePreserving}
\shape(\rho\cdot\sigma) = \shape(\sigma).
\end{equation}
By \cref{thm:Brauer} and \cref{thm:rhoadesPMatrix} the
action of $\rho$ and rotation are isomorphic.
Therefore we have a bijection $\phi:\symS_n \to \symS_n$ with
\begin{equation}\label{eq:isomorphism}
\phi(\rot\sigma) = \rho\cdot \phi(\sigma).
\end{equation}
Defining $\st(\sigma)\coloneqq \shape(\phi(\sigma))$, we obtain
\[
\st(\rot \sigma) = \shape(\phi(\rot\sigma)) \stackrel{\eqref{eq:isomorphism}}{=} \shape(\rho\cdot\phi(\sigma)) \stackrel{\eqref{eq:shapePreserving}}{=} \shape(\phi(\sigma)) = \st(\sigma).
\]
Finally we have
\[
\symS_n^{\lambda} = \phi^{-1}(R_\lambda),
\]
yielding the last statement.
\end{proof}

\begin{remark}
  It is natural to ask whether for $\lambda \vdash n$ we have
  \[
    \#\{ \sigma \in R_\lambda : \rot^d (\sigma) = \sigma \} = f^\lambda(\xi^d)^2,
  \]
  for all $d\in \setN$ and a primitive $n^\thsup$ root of unity
  $\xi$.  In this case, the \emph{subset cyclic sieving} technique of
  P.~Alexandersson, S.~Linusson \&
  S.~Potka~\cite[Prop. 29]{AlexanderssonLinussonPotka2019} would
  imply the non-skew, $m=2$ case of
  \cref{prop:positive-cyclic-sieving}.
  However, this fails already for $\lambda=(2,1)$; we have that
  $ R_{\lambda} = \{132,213,231,312\}$ and $\rot$ fixes $231$ and
  $312$, but $f^{\lambda}(q)^2 = q^2(1+q)^2$ evaluates to $1$ at
  $q=\exp(2\pi i/3)$.
\end{remark}

We now turn to the connection with the invariants of tensor powers of
the adjoint representation of $\GL_r$, which is the original
motivation for this article.

Let $V$ be an $r$-dimensional complex vector space and let
$\gl_r = \End(V)$ be the adjoint representation
$\GL_r\to\End(\gl_r)$, $A\mapsto T A T^{-1}$.  Recall that the space
of $\GL_r$-invariants of the $n^\thsup$ tensor power of $\gl_r$ is, as a representation of the symmetric group $\symS_n$,
\[
  \left(\gl_r^{\otimes n}\right)^{\GL_r} %
  = \Hom_{\GL_r}\left(\gl_r^{\otimes n}, \setC\right) %
  \cong \Hom_{\GL_r}\left(\left(V\otimes V^*\right)^{\otimes n}, \setC\right) %
  \cong\End_{\GL_r}(V^{\otimes n}).
\]
A basis for this space can be indexed by J.~Stembridge's alternating
tableaux:
\begin{definition}[\cite{Stembridge1987}]\label{defn:alternating}
  A \defin{staircase} is a dominant weight of $\GL_r$, that is, a
  vector in $\setZ^r$ with weakly decreasing entries.  A
  \defin{$\gl_r$-alternating tableau} $\alt$ of length $n$ (and
  weight zero) is a sequence of staircases
  \[
  \alt=(\emptyset\!=\!\wgt^0,\wgt^1,\ldots,\wgt^{2n}\!=\!\emptyset)
  \]
  such that
  \begin{itemize}\setlength{\itemsep}{0pt}
  \item[] for even $i$, $\wgt^{i+1}$ is obtained from $\wgt^i$ by
    adding $1$ to an entry, and
  \item[] for odd $i$, $\wgt^{i+1}$ is obtained from $\wgt^i$ by
    subtracting $1$ from an entry.
  \end{itemize}
  The set of $\gl_r$-alternating tableaux of length $n$ is denoted by
  $\Alt{r}{n}$.
\end{definition}

B.~Westbury defined a natural action, \defin{promotion}, of the cyclic
group of order $n$ on the set of so called invariant words of any
finite crystal, in particular alternating tableaux of length $n$,
generalizing Schützenberger's promotion on rectangular standard Young
tableaux.  We refrain from giving a definition here and refer to
S.~Pfannerer, M.~Rubey \& B.~Westbury~\cite{PfannererRubeyWestbury2020} instead.

For our purposes, it is enough to relate promotion to an action on
the $\GL_r$-invariants of the $n^\thsup$ tensor power of $\gl_r$.
To do so, note that the symmetric group $\symS_n$ acts on
$\gl_r^{\otimes n}$ by permuting tensor positions, and therefore also
on the space of invariants.  It turns out that the action of the long
cycle $\longcycle\in\symS_n$ plays a special role:
\begin{theorem}[{\cite[Sec. 6.3]{Westbury2016}}]\label{thm:pr-rot}
  There is a basis of
  $\left(\gl_r^{\otimes n}\right)^{\GL_r}$ which is preserved by the
  action of the long cycle.  Moreover, this action is isomorphic to
  the action of promotion on the set of alternating tableaux.
\end{theorem}

Note that B.~Westbury's theorem only asserts the existence of the basis,
no explicit construction is known.  The main result of this section
is the following refinement of his assertion.
\begin{theorem}\label{thm:existence-pr-rot}
  Let
  $\symS_n^{(r)}\coloneqq\{\pi\in\symS_n\mid \ell(\st(\pi)) \leq
  r\} = \bigcup\limits_{\substack{\lambda\vdash n\\\ell(\lambda)\leq r}}\symS_n^{\lambda}$.  Then there exists a bijection
  \[
    \Perm:\Alt{r}{n}\to\symS_n^{(r)}\quad \text{with}\quad\Perm\circ\pr= \rot\circ\Perm.
  \]
  for $1\leq r\leq n$.
\end{theorem}
\begin{remark}
  In particular, there is an injection
  $\iota:\Alt{r}{n}\to\Alt{r+1}{n}$ such that
  $\pr\iota(\alt) = \iota(\pr\alt)$.  This answers a question posed
  by S.~Pfannerer, M.~Rubey \& B.~Westbury~\cite[rmk.~3.9]{PfannererRubeyWestbury2020}.
\end{remark}

Let us remark that for large dimension such a bijection is known:
\begin{theorem}[{\cite{PfannererRubeyWestbury2020}}]\label{thm:PRW}
  For $r\geq n$, there is an (explicit) bijection
  \[
  \Perm:\Alt{r}{n}\to\symS_n\quad \text{with}\quad\Perm\circ\pr=
  \rot\circ\Perm.
  \]
\end{theorem}

As a first step, we compute a decomposition of the space of
invariants as a direct sum of tensor squares of Specht modules.
\begin{lemma}[\cite{RubeyWestbury2014}]\label{thm:invariants-character}
  Let $\gl_r$ be the adjoint representation of $\GL_r$, and, given a
  partition $\lambda\vdash n$, let $S_\lambda$ be the corresponding
  irreducible representation of the symmetric group.  Then there is
  an isomorphism
  \[
  \left(\gl_r^{\otimes n}\right)^{\GL_r}\cong%
  \bigoplus_{\substack{\lambda\vdash n\\\ell(\lambda)\leq r}}
  S_\lambda\otimes S_\lambda.
  \]
\end{lemma}
\begin{proof}
  Let $V$ be the vector representation of $\GL_r$.  Schur--Weyl
  duality asserts that there is an isomorphism
  \[
  V^{\otimes r} \cong %
  \bigoplus_{\substack{\lambda\vdash n\\\ell(\lambda)\leq r}} %
  V_\lambda\otimes S_\lambda,
  \]
  where $V_\lambda$ is an irreducible representation of $\GL_r$.

  Recall that, by Schur's lemma, $\Hom_{\GL_r}(V_\lambda, V_\mu)$
  contains only the zero map if $\lambda\neq\mu$, and all scalar
  multiples of the identity otherwise.  Thus,
  \begin{align*}
    \left(\gl_r^{\otimes n}\right)^{\GL_r} %
    &\cong \End_{\GL_r}(V^{\otimes n})\\
    &\cong \End_{\GL_r}(%
      \bigoplus_{\substack{\lambda\vdash n\\\ell(\lambda)\leq r}} %
    V_\lambda\otimes S_\lambda)\\\{\text{by Schur's Lemma}\}
    &\cong \bigoplus_{\substack{\lambda\vdash n\\\ell(\lambda)\leq r}} %
    \End(S_\lambda)\\
    &\cong \bigoplus_{\substack{\lambda\vdash n\\\ell(\lambda)\leq r}} %
    S_\lambda \otimes S_\lambda.
  \end{align*}
\end{proof}
\begin{proof}[Proof of \cref{thm:existence-pr-rot}]
  By \cref{thm:rootEvaluationGivesCharacter} and \cref{cor:rotationInvatiantStat} we obtain that the character of
  $\symS_n^{\lambda}$ equals the character of
  $\bigoplus_{\substack{\lambda\vdash n\\\ell(\lambda)\leq r}}
  S_\lambda\otimes S_\lambda\tocyclic$.
Summing over all partitions $\lambda$ of length at most $r$, we obtain the character of
  \[
\left(\gl_r^{\otimes n}\right)^{\GL_r}\tocyclic %
\cong \bigoplus_{\substack{\lambda\vdash n\\\ell(\lambda)\leq r}} %
S_\lambda\otimes S_\lambda\tocyclic,
\]
by \cref{thm:invariants-character}.
Therefore, by Brauer's permutation lemma (\cref{thm:Brauer}) and Westbury's \cref{thm:pr-rot},
the cyclic group actions
\[
(\pr, \Alt{r}{n}) \cong (\rot, \symS_n^{(r)})
\]
are isomorphic.
\end{proof}

\bibliographystyle{alphaurl}
\bibliography{bibliography}

\end{document}